\documentclass[12pt]{amsart}
\usepackage{geometry}
\usepackage{titletoc}
\usepackage[colorlinks, citecolor=red, linkcolor=blue, hyperindex]{hyperref}
\usepackage{euscript, eufrak, verbatim, mathrsfs}
\usepackage[psamsfonts]{amssymb}
\usepackage{bbm}
\usepackage{graphicx}
\usepackage{float}
\usepackage{float, tikz}
\usepackage[all, cmtip]{xy}
\usepackage{upref, xcolor, dsfont}
\usepackage{amsfonts, amsmath, amstext, amsbsy, amsopn, amsthm}
\usepackage{enumerate}
\usepackage{url}
\usepackage{mathtools}
\usepackage{bookmark}
\usepackage{euscript}
\usepackage{helvet}         % selects\textbf{\textbf{•}} Helvetica as sans-serif font
\usepackage{courier}        % selects Courier as typewriter font
\usepackage{type1cm}        % activate if the above 3 fonts are
\usepackage{multicol}        % used for the two-column index
\usepackage[bottom]{footmisc}% places footnotes at page bottom

\newtheorem{theorem}{Theorem}[section]
\newtheorem*{theorem*}{Theorem}
\newtheorem{lemma}[theorem]{Lemma}

\newtheorem{corollary}[theorem]{Corollary}

\newtheorem*{comment*}{Comment}
\newtheorem*{definition*}{Definition}
\newtheorem*{remark*}{Remark}

\newtheorem*{observation*}{Observation}

\newtheorem*{assumption*}{Assumption}

\theoremstyle{definition}
\newtheorem{definition}{Definition}[section]
\newtheorem{question}{Question}

\theoremstyle{remark}
\newtheorem{remark}{Remark}[section]
\newcommand{\Var}{\mathrm{Var}}

\newcommand{\supp}{\mathrm{supp}}

\newcommand{\Conf}{\mathrm{Conf}}

\begin{document}

\title[Fluctuations of the process of moduli]{Fluctuations of the process of moduli for the Ginibre and hyperbolic ensembles}

\author
{Alexander I. Bufetov}
\address{Alexander I. Bufetov: CNRS, Aix-Marseille Universit{\'e}, Centrale Marseille, Institut de Math{\'e}matiques de Marseille, UMR7373, 39 Rue F. Joliot Curie 13453, Marseille, France; 
\newline Steklov Mathematical Institute of RAS, Moscow, Russia; 
\newline Institute for Information Transmission Problems, Moscow, Russia.}
\email{alexander.bufetov@univ-amu.fr; bufetov@mi-ras.ru}

\author
{David Garc\'ia-Zelada}
\address
{David Garc\'ia-Zelada: Laboratoire de Probabilit{\'e}s, Statistique et Mod{\'e}lisation, UMR CNRS 8001, Sorbonne Universit{\'e}, 4 Place Jussieu, 75005 Paris, France}
\email{david.garcia-zelada@sorbonne-universite.fr}

\author
{Zhaofeng Lin}
\address{Zhaofeng Lin: Shanghai Center for Mathematical Sciences, Fudan University, Shanghai, 200438, China.}
\email{zflin18@fudan.edu.cn}

\thanks{A. I. Bufetov's research has received funding from the European Research Council (ERC) under the European Union's Horizon 2020 research and innovation programme, grant agreement No 647133 (ICHAOS), as well as from the ANR grant ANR-18-CE40-0035 REPKA. Z. Lin is supported by grants NSFC 11722102 and NSFC 12026250.}

\begin{abstract}
	We investigate the point process of moduli of the Ginibre and hyperbolic ensembles. We show that far from the origin and at an appropriate scale, these processes exhibit Gaussian and Poisson fluctuations. Among the possible Gaussian fluctuations, we can find white noise but also fluctuations with non-trivial covariance at a particular scale.
\end{abstract}



\maketitle

\section{Formulation of the main results}

The main results of this paper, Theorems~\ref{Extreme Hyperbolic Normality}-\ref{Ginibre Poisson}, establish limit theorems for additive statistics of the Ginibre ensemble and the hyperbolic ensembles, introduced by Krishnapur, including the determinantal point process with the Bergman kernel, which, by the Peres-Vir\'ag theorem, is the zero set of the Gaussian analytic function on the unit disc.

In this section, we begin by recalling the notion of determinantal point process, which are point processes where the correlation functions take the form of a determinant. Afterwards, the specific examples we are interested in, namely the Ginibre point process and the hyperbolic ensembles, are discussed. Finally, we state the main results of this note, Theorems~\ref{Extreme Hyperbolic Normality}-\ref{Ginibre Poisson}.

\subsection{Determinantal point process}
Let $X$ be a locally compact Polish space and $\mathcal{B}_0(X)$ the collection of all pre-compact Borel subsets of $X$. We shall denote by $\Conf(X)$, the space of all locally finite configurations over $X$, that is,
\begin{equation*}
	\begin{split}
		\Conf(X):=\big\{\xi=\textstyle\sum_{i}\delta_{x_i}\,\big|\,\,\forall i,\,x_i\in X\, \text{and $\xi(\Delta)<\infty$ for all $\Delta\in\mathcal{B}_0(X)$}\big\}.
	\end{split}
\end{equation*}
We shall consider this set endowed with the vague topology, i.e., the weakest topology on $\Conf(X)$ such that for any compactly supported continuous function $f$ on $X$, the map $\Conf(X)\ni\xi\mapsto\int_{X}f\mathrm d\xi$ is continuous. It can be seen that the configuration space $\Conf(X)$ equipped with the vague topology turns out to be a Polish space. Additionally, it can be seen that the Borel $\sigma$-algebra $\mathcal{F}$ on $\Conf(X)$ is generated by the cylinder sets $C_n^\Delta=\big\{\xi\in\Conf(X)\,|\,\,\xi(\Delta)=n\big\}$, where $n\in\mathbb{N}=\{0,1,2,\cdots\}$ and $\Delta\in\mathcal{B}_0(X)$.  Finally, we will say that a measurable map
\begin{equation*}
	\begin{split}
		\mathscr{X}:(\Omega,\mathcal{F}(\Omega),\mathbb{P})\to(\Conf(X),\mathcal{F})
	\end{split}
\end{equation*}
is a point process on $X$, where $(\Omega,\mathcal{F}(\Omega),\mathbb{P})$ is any probability space. For futher background, see \cite{DV, KK, Le}.

A point process $\mathscr{X}$ is called simple if it almost surely assigns at most measure one to singletons. In the simple case, $\mathscr{X}$ can be identified with a random discrete subset of $X$ and for any Borel set $\Delta$ on $X$, the number $\mathscr{X}(\Delta) \in \mathbb N \cup \{\infty\}$ represents the number of points of this discrete subset that fall in $\Delta$.
So, for instance, we will use the notation 
$\{T(x): x \in \mathscr{X}\}$ instead of
the usual pushforward notation $T_*\mathscr X$
\hspace{-2mm}, 
for simplicity.

Determinantal point processes have been introduced by Odile Macchi \cite{Ma} in the seventies. We recall the definition. Let $\mu$ be a Radon measure on $X$ and let $K:X\times X\to\mathbb{C}$ be a measurable function. A simple point process $\mathscr{X}$ is called determinantal on $X$ associated to the kernel $K$ with respect to the reference measure $\mu$ if, for every $k\in\mathbb{N}_+$ and any family of mutually disjoint subsets $\Delta_1,\Delta_2,\cdots,\Delta_k\in\mathcal{B}_0(X)$,
\begin{equation}\label{DPP-definition}
	\begin{split}
		\mathbb{E}\Big[\prod_{i=1}^{k}\mathscr{X}(\Delta_i)\Big]=\int_{\Delta_1\times\cdots\times\Delta_k} \det\big[K(x_i,x_j)\big]_{1\leq i,j\leq k}\mathrm{d}\mu(x_1)\cdots \mathrm{d}\mu(x_k).
	\end{split}
\end{equation}
See, e.g., \cite{Bo, Bu, BD, BQ, HKPV, PV, ST, So, Sos} for further background of determinantal point process.

The moments of the linear statistics $\sum_{x\in\mathscr{X}}f(x):=\int_{X}f\mathrm{d}\mathscr{X}$ under a determinantal point process can be calculated from \eqref{DPP-definition}. For instance, when the kernel $K$ is Hermitian and satisfies the reproducing property, i.e., it represents an orthogonal projection, then
\begin{equation}\label{Expected value}
	\begin{split}
		\mathbb{E}\Big[\sum_{x\in\mathscr{X}}f(x)\Big]&=\int_{X}f(x)K(x,x)\mathrm{d}\mu(x),
	\end{split}
\end{equation}
and
\begin{equation}\label{Variance}
	\begin{split}
		\Var\Big(\sum_{x\in\mathscr{X}}f(x)\Big)=\frac{1}{2}\int_{X^2}\big[f(x)-f(y)\big]^2\big|K(x,y)\big|^2\mathrm{d}\mu(x)\mathrm{d}\mu(y).
	\end{split}
\end{equation}
See, e.g., \cite[Proposition 4.1]{Gh}, \cite[Lemma 8.5]{GP} and \cite[formulas (4) and (5)]{Sos}.

\subsection{The Ginibre ensemble}
Consider the Gaussian measure $\mu$ on the whole plane $\mathbb{C}$ given by
\begin{equation*}
	\begin{split}
		\mathrm{d}\mu(z)=\frac{1}{\pi}e^{-|z|^2}\mathrm{d}m(z),
	\end{split}
\end{equation*}
where $\mathrm{d}m$ is the usual Lebesgue measure. We also need to consider the space  $\mathcal{O}(\Lambda)$ of holomorphic functions on an open set $\Lambda\subset\mathbb{C}$.

In the finite-dimensional setting, the Ginibre ensemble was introduced by Ginibre \cite{Gi} as a model based on the eigenvalues of non-Hermitian random matrices, and the infinite Ginibre ensemble is obtained as a weak limit of these finite-dimensional point processes. The infinite Ginibre ensemble can be defined as follows.

\begin{definition}[Ginibre ensemble]
	The \emph{Ginibre ensemble} $\mathcal{G}$ is the determinantal point process associated to the Fock kernel, i.e., the kernel of the orthogonal projection of $L^2(\mathbb{C},\mu)$ onto $L^2(\mathbb{C},\mu)\cap\mathcal{O}(\mathbb{C})$, with respect to the reference measure $\mu$.
\end{definition}

Equivalently, the Ginibre ensemble is the point process $\mathcal{G}$ on $\mathbb{C}$ such that for any pairwise disjoint measurable subsets
$\Delta_1, \Delta_2, \dots, \Delta_k$ of $\mathbb{C}$, we have that
\begin{equation*}
	\begin{split}
		\mathbb{E}\Big[\prod_{i=1}^{k}\mathcal{G}(\Delta_i)\Big]=\int_{\Delta_1\times\cdots\times\Delta_k}\det\big[K_{\mathcal{G}}(z_i,z_j)\big]_{1\leq i,j\leq k}\mathrm{d}m(z_1)\cdots\mathrm{d}m(z_k),
	\end{split}
\end{equation*}
where $K_{\mathcal{G}}: \mathbb{C}\times\mathbb{C}\to\mathbb{C}$ is given by
\begin{equation*}
	\begin{split}
		K_{\mathcal{G}}(z,w)=\frac{1}{\pi}\sum_{n=0}^\infty\frac{z^n\overline{w}^n}{n!}e^{-\frac{|z|^2}{2}}e^{-\frac{|w|^2}{2}}=\frac{1}{\pi}e^{z\overline{w}-\frac{|z|^2}{2}-\frac{|w|^2}{2}}.
	\end{split}
\end{equation*}

\subsection{The hyperbolic ensembles}
For each $\alpha>0$, consider the probability measure $\mu_{\alpha}$ on the unit disc $\mathbb{D}$ given by
\begin{equation*}
	\begin{split}
		\mathrm{d}\mu_{\alpha}(z)=\frac{\alpha}{\pi}(1-|z|^2)^{\alpha-1}\mathrm{d}m(z)
	\end{split}
\end{equation*}
and recall that $\mathcal O(\mathbb D)$ denotes
the space of holomorphic functions on $\mathbb D$.
\begin{definition}[Hyperbolic ensemble]
	For $\alpha>0$, the $\alpha$-\emph{hyperbolic ensemble} $\mathcal{H}_{\alpha}$ is the determinantal point process 
	associated to the $\mu_{\alpha}$-weighted Bergman kernel, i.e.,
	the kernel of the orthogonal projection of $L^2(\mathbb{D},\mu_{\alpha})$ onto $L^2(\mathbb{D},\mu_{\alpha})\cap\mathcal{O}(\mathbb{D})$, 
	with respect to the reference measure $\mu_\alpha$.
\end{definition}

Equivalently, the $\alpha$-hyperbolic ensemble is the point process $\mathcal{H}_{\alpha}$ on $\mathbb{D}$ such that for any pairwise disjoint measurable subsets
$\Delta_1, \Delta_2, \dots, \Delta_k$ of $\mathbb{D}$, we have that
\begin{equation*}
	\begin{split}
		\mathbb{E}\Big[\prod_{i=1}^{k}\mathcal{H_{\alpha}}(\Delta_i)\Big]=\int_{\Delta_1\times\cdots\times\Delta_k}\det\big[K_{\mathcal{H}_{\alpha}}(z_i,z_j)\big]_{1\leq i,j\leq k}\mathrm{d}m(z_1)\cdots\mathrm{d}m(z_k),
	\end{split}
\end{equation*}
where $K_{\mathcal{H}_{\alpha}}: \mathbb{D}\times\mathbb{D}\to\mathbb{C}$ is given by
\begin{equation*}
	\begin{split}
		K_{\mathcal{H}_{\alpha}}(z,w)=\frac{1}{\pi}\sum_{n=0}^\infty k_n^{(\alpha)}z^n\overline{w}^n(1-|z|^2)^{\alpha-1}(1-|w|^2)^{\alpha-1}=\frac{\alpha}{\pi}\frac{(1-|z|^2)^{\alpha-1}(1-|w|^2)^{\alpha-1}}{(1-z\overline{w})^{\alpha+1}},
	\end{split}
\end{equation*}
while
\begin{equation*}
	\begin{split}
		k_n^{(\alpha)}=\frac{\alpha(\alpha+1)\cdots(\alpha+n)}{n!}=\frac{\Gamma(\alpha+n+1)}{\Gamma(\alpha)\Gamma(n+1)}.
	\end{split}
\end{equation*}

For $\alpha=1$, Peres and Vir\'ag \cite{PV} showed that the zeros of the random analytic function
\begin{equation*}
	\begin{split}
		\phi(z)=a_0+a_1z+a_2z^2+\cdots,
	\end{split}
\end{equation*}
where $a_k$, $k\geq0$, are independent and identically distributed standard complex Gaussian random variables, follow the law of the $1$-hyperbolic ensemble $\mathcal{H}_1$. Krishnapur \cite{Kr} extended the result of Peres and Vir\'ag to positive integer $\alpha=m$, showing that, if $G_k$, $k\geq0$, are independent and identically distributed $m\times m$ matrices, each with independent and identically distributed standard complex Gaussian entries, then the zeros of the random analytic function
\begin{equation*}
	\begin{split}
		\Phi(z)=\det(G_0+G_1z+G_2z^2+\cdots)
	\end{split}
\end{equation*}
follow the law of the $m$-hyperbolic ensemble 
$\mathcal H_m$.

\subsection{Main results}
Recall that the unit disc $\mathbb{D}$ endowed with the metric $\mathrm{d}\mu_{\alpha}$ is the Poincar\'e model for the Lobachevsky plane. Our results involve the point process formed by $\{|z|: z\in\mathcal{G}\}$ and the one formed by $\{|z|_{\mathrm{h}}: z\in\mathcal{H}_{\alpha}\}$, $\alpha>0$, where
\begin{equation*}
	\begin{split}
		|z|_{\mathrm{h}} = \log{\frac{1+|z|}{1-|z|}}
	\end{split}
\end{equation*}
is the hyperbolic distance from $z$ to the origin.

By \cite[Theorem~4.7.1]{HKPV}, which was first noticed
by Kostlan \cite[Lemma~1.4]{Ko} in the case of a finite number of particles, the point process $\{|z|: z\in\mathcal{G}\}$ follows the law of $\{\rho_n: n\in\mathbb{N}\}$, where
$(\rho_n)_{n\geq0}$ is a family of non-negative independent
random variables such that
\begin{equation*}
	\begin{split}
		\rho_n\sim\frac{2r^{2n+1}e^{-r^2}}{n!}\mathrm{d}r.
	\end{split}
\end{equation*}
Similarly, from \cite[Theorem~4.7.1]{HKPV}, see also \cite{Kr} and \cite{PV}, for $\alpha>0$, the point process $\{|z|: z\in\mathcal{H}_{\alpha}\}$ follows the law of $\{\rho_n^{(\alpha)}: n\in\mathbb{N}\}$, where $(\rho_n^{(\alpha)})_{n\geq0}$ is a family of independent random variables taking values in $[0,1]$ such that
\begin{equation*}
	\begin{split}
		\rho_n^{(\alpha)}\sim2\frac{\Gamma(\alpha+n+1)}{\Gamma(\alpha)\Gamma(n+1)}r^{2n+1}(1-r^2)^{\alpha-1}\mathrm{d}r.
	\end{split}
\end{equation*}

The asymptotic behavious of $\{|z| - R: z\in\mathcal{G}\}$ and of $\{|z|_{\mathrm{h}} - R: z\in\mathcal{H}_{\alpha}\}$, $\alpha>0$, as $R$ goes to infinity and under different scalings is described
in the theorems below. Theorem \ref{Extreme Hyperbolic Normality}
and \ref{Extreme Ginibre Normality} deal with a convergence
of the normalized process
towards a Gaussian field whose covariance kernel
is not the one of the white noise. Theorem \ref{White Noise Hyperbolic} and Theorem \ref{White Noise Ginibre}
deal with the intermediate
case of a convergence towards the white noise.
Finally, the last two theorems, Theorem  \ref{Hyperbolic Poisson}
and Theorem  \ref{Ginibre Poisson},
deal with the extreme case of the
convergence towards a homogeneous Poisson process.

\subsubsection{Convergence towards a non-trivial Gaussian limit}
\begin{theorem}\label{Extreme Hyperbolic Normality}
	Denote $C_R^{(\alpha)}=\alpha e^R/8$ for each $\alpha>0$. Then for any bounded measurable and compactly supported function $f: \mathbb{R}\to\mathbb{R}$,
	\begin{equation*}
		\begin{split}
			\frac{1}{\sqrt{C_R^{(\alpha)}}}\sum_{z\in\mathcal{H}_{\alpha}}f\big(|z|_{\mathrm{h}}-R\big)-2\sqrt{C_R^{(\alpha)}}\int_{\mathbb{R}}f(x)e^{x}\mathrm{d}x\xrightarrow[R\to+\infty]{\mathrm{law}}\mathcal{N}\big(0,V_f^{(\alpha)}\big),
		\end{split}
	\end{equation*}
	where
	\begin{equation*}
		\begin{split}
			V_f^{(\alpha)}=\frac{1}{B(\alpha,\alpha+1)}\int_{\mathbb{R}^2}\big[f(x)-f(y)\big]^2\frac{e^{(\alpha+1)(x+y)}}{\left(e^{x}+e^{y}\right)^{2\alpha+1}}\mathrm{d}x\mathrm{d}y.
		\end{split}
	\end{equation*}
\end{theorem}

\begin{theorem}\label{Extreme Ginibre Normality}
	For any bounded measurable and compactly supported function $f: \mathbb{R}\to\mathbb{R}$,
	\begin{equation*}
		\begin{split}
			\frac{1}{\sqrt{R}}\sum_{z\in\mathcal{G}}f\big(|z|-R\big)-2\sqrt{R}\int_{\mathbb{R}}f(x)\mathrm{d}x\xrightarrow[R\to+\infty]{\mathrm{law}}\mathcal{N}(0, V_f),
		\end{split}
	\end{equation*}
	where
	\begin{equation*}
		\begin{split}
			V_f=\frac{1}{\sqrt{\pi}}\int_{\mathbb{R}^2}\big[f(x)-f(y)\big]^2e^{-(x-y)^2}\mathrm{d}x\mathrm{d}y.
		\end{split}
	\end{equation*}
\end{theorem}

\subsubsection{Convergence towards white noise}

\begin{theorem}\label{White Noise Hyperbolic}
	Suppose that $a_R$ satisfies $1\ll a_R\ll e^R$ as $R\to+\infty$. Denote \break $C_R^{(\alpha)}=\alpha e^R/8$ for each $\alpha>0$. Then for any bounded measurable and compactly supported function $f: \mathbb{R}\to\mathbb{R}$,
	\begin{equation*}
		\begin{split}
			\sqrt{\frac{a_R}{C_R^{(\alpha)}}}\sum_{z\in\mathcal{H}_{\alpha}}f\big(a_R(|z|_{\mathrm{h}}-R)\big)-2\sqrt{\frac{C_R^{(\alpha)}}{a_R}}\int_{\mathbb{R}}f(x)\mathrm{d}x\xrightarrow[R\to+\infty]{\mathrm{law}}\mathcal{N}\Big(0,2\int_{\mathbb{R}}f^2(x)\mathrm{d}x\Big).
		\end{split}
	\end{equation*}
\end{theorem}

\begin{theorem}\label{White Noise Ginibre}
	Suppose that $a_R$ satisfies $1\ll a_R\ll R$ as $R\to+\infty$. Then for any bounded measurable and compactly supported function $f: \mathbb{R}\to\mathbb{R}$,
	\begin{equation*}
		\begin{split}
			\sqrt{\frac{a_R}{R}}\sum_{z\in\mathcal{G}}f\big(a_R(|z|-R)\big)-2\sqrt{\frac{R}{a_R}}\int_{\mathbb{R}}f(x)\mathrm{d}x\xrightarrow[R\to+\infty]{\mathrm{law}}\mathcal{N}\Big(0,2\int_{\mathbb{R}}f^2(x)\mathrm{d}x\Big).
		\end{split}
	\end{equation*}
\end{theorem}

Theorems~\ref{Extreme Hyperbolic Normality}-\ref{White Noise Ginibre} will be obtained as corollaries of the central limit theorem of Soshnikov.

\subsubsection{Convergence towards a Poisson point process}
\begin{theorem}\label{Hyperbolic Poisson}
	Let $\mathcal{P}_{\alpha/4}$ be the Poisson point process on $\mathbb{R}$ with constant intensity $\alpha/4$, $\alpha>0$. Then,
	\begin{equation*}
		\begin{split}
			\big\{e^{R}(|z|_{\mathrm{h}}-R):z\in\mathcal{H}_{\alpha}\big\}\xrightarrow[R\to+\infty]{\mathrm{law}}\mathcal{P}_{\alpha/4}.
		\end{split}
	\end{equation*}
\end{theorem}

\begin{theorem}\label{Ginibre Poisson}
	Let $\mathcal{P}_2$ be the Poisson point process on $\mathbb{R}$ with constant intensity $2$. Then,
	\begin{equation*}
		\begin{split}
			\big\{R(|z|-R):z\in\mathcal{G}\big\}\xrightarrow[R\to+\infty]{\mathrm{law}}\mathcal{P}_2.
		\end{split}
	\end{equation*}
\end{theorem}

\begin{remark}
	An analog of the theorems above is the case of an independent and identically distributed sequence $(X_i)_{i \geq 1}$ of real random variables that follows a probability distribution $\mu$. The analogue of Theorem~\ref{Extreme Hyperbolic Normality} and Theorem~\ref{Extreme Ginibre Normality} is the classical central limit theorem that tells us that, for any compactly supported function $f:\mathbb R \to \mathbb R$, we have
	\[\frac{\sum_{i=1}^n f(X_i) - n \int \hspace{-1mm} f \mathrm d \mu}{\sqrt n} \xrightarrow[n \to +\infty]{\mathrm{law}}  \mathcal N \left(0,\frac{1}{2}\int_{\mathbb{R}^2} [f(x)-f(y)]^2 \mathrm d \mu(x) \mathrm d\mu(y) \right).\]
	If these random variables admit a density $\rho$ which is continuous at $0$,  we have the convergence towards a Poisson point process, the analogue of Theorem~\ref{Hyperbolic Poisson} and Theorem~\ref{Ginibre Poisson},
	\[\{n X_i : 1 \leq i \leq n\} \xrightarrow[n \to +\infty]{\mathrm{law}}  \mathcal P_{\rho(0)}.\]
	Notice that we can also see, at intermediate scalings,
	a convergence towards the white noise for
	the centered linear statistics.
\end{remark}

\section{The hyperbolic process proof of Theorem~\ref{Extreme Hyperbolic Normality}}
To study the limiting behaviour of the linear statistics
\begin{equation*}
	\begin{split}
		\sum_{z\in\mathcal{H}_{\alpha}}f\big(|z|_{\mathrm{h}}-R\big),
	\end{split}
\end{equation*}
we start by understanding
the asymptotics of its expected value. Later, we study its limiting variance and conclude by Soshnikov's central limit theorem \cite[Theorem 1]{Sos}.

We prepare two lemmas that contain some bounds that we will use.

\begin{lemma}
\label{lem:BoundExp}
There exists two constants $C, M>0$ such that,
for every $y \geq M$ and $t \geq 0$,
\[\left|\left(1-\frac{2}{y+1}\right)^{2\lfloor t y \rfloor}
- e^{-4t} \right| \leq \frac{C e^{-t/C}}{y} ,\]
where $\lfloor t y \rfloor$ is the biggest integer that does not exceed $ty$.
\end{lemma}
\begin{proof}[Proof of Lemma \ref{lem:BoundExp}]
We can use Lagrange's mean value theorem 
to control the difference
\begin{equation*}
	\begin{split}
		&\quad\,\left|\Big(1-\frac{2}{y+1}\Big)^{2{\lfloor ty\rfloor}+1}-e^{-4t}\right|\\
		&\leq\left|\Big[\Big(1+\frac{1}{\frac{y+1}{2}-1}\Big)^{\frac{y+1}{2}}\Big]^{-\frac{2(2\lfloor ty\rfloor+1)}{y+1}}-e^{-\frac{2(2\lfloor ty\rfloor+1)}{y+1}}\right|+\left|e^{-\frac{2(2\lfloor ty\rfloor+1)}{y+1}}-e^{-4t}\right|\\
		&=\frac{2(2\lfloor ty\rfloor+1)}{y+1}
		\frac{1}{\xi^{\frac{2(2\lfloor ty\rfloor+1)}{y+1}+1}}
		\left[\Big(1+\frac{1}{\frac{y+1}{2}-1}\Big)^{\frac{y+1}{2}}-e\right]
		+e^{\eta}
		\left|\frac{2(2\lfloor ty\rfloor+1)}{y+1}-4t\right| ,
	\end{split}
\end{equation*}
where $\xi\geq e$ and $\eta \leq -A (t+1)$ for some fixed positive constant $A<4$
and $y$ large enough.  The asymptotics of each of
these terms as $y$ goes to $+\infty$ gives us the result. 

Now, notice that there exist two constants $C,M>0$ such that,
for every $y\geq M$ and $t \geq 0$, we have the 
bounds \vspace{-3mm}

\begin{align*}
& \bullet	\displaystyle \frac{2(2\lfloor ty\rfloor+1)}{y+1}
\leq C (t + 1), 
& &	\bullet
	\frac{1}{\xi^{\frac{2(2\lfloor ty\rfloor+1)}{y+1}+1}}
\leq Ce^{-t/C},														\\
&  \bullet
	 \left[\Big(1+\frac{1}{\frac{y+1}{2}-1}\Big)^{\frac{y+1}{2}}-e\right] \leq \frac{C}{y}, 							
& & \bullet e^\eta \leq  C e^{-t/C},
		\\[1ex]
& \bullet \left|\frac{2(2\lfloor ty\rfloor+1)}{y+1}-4t\right|
	\leq C(t+1).
\end{align*}
The lemma follows by noticing that
$(t+1)^2 e^{-t/C} \leq \widetilde C e^{-t/\widetilde C}$
for some $\widetilde C>0$.
\end{proof}

For the next lemma, we recall the notation
\[k_n^{(\alpha)} = \frac{\Gamma(\alpha + n + 1)}{\Gamma(\alpha)
\Gamma(n+1)}.\]

\begin{lemma}
\label{lem:Stirling}
There exists a constant $C>0$ such that
\[\left|k_{\lfloor y \rfloor } - \frac{y^\alpha}{\Gamma(\alpha)}
 \right|\leq\left\{\begin{array}{ll}
			C,& y \in [0,1)\\
			C y^{\alpha-1},& y \in [1,+\infty)
		\end{array}\right.,
\]
where $\lfloor y \rfloor$ is the biggest integer that does not exceed $y$.

\end{lemma}
\begin{proof}[Proof of Lemma \ref{lem:Stirling}]

This is a consequence of Stirling's series
that tells us that
\begin{equation*}
	\begin{split}
		k_n^{(\alpha)}=\frac{\Gamma(\alpha+n+1)}{\Gamma(\alpha)\Gamma(n+1)}=\frac{n^{\alpha}}{\Gamma(\alpha)}+O(n^{\alpha-1})\quad\mathrm{as}\,\,n\to\infty.
	\end{split}
\end{equation*}
The piecewise inequality comes from
the fact that $k_0^{(\alpha)} \neq 0$.
\end{proof}

With these lemmas, it is easier to understand
the asymptotics of the expected value.

\subsection{Expected value calculation for Theorem~\ref{Extreme Hyperbolic Normality}}\label{Expected value calculation for Extreme Hyperbolic Normality}

We will show a bit more than it is needed about the expected value
asymptotics. We will see that
\begin{equation*}
	\begin{split}
		\mathbb{E}\Big[\sum_{z\in\mathcal{H}_{\alpha}}f\big(|z|_{\mathrm{h}}-R\big)\Big]=2C_R^{(\alpha)}\int_{\mathbb{R}}f(x)e^{x}\mathrm{d}x+O(1).
	\end{split}
\end{equation*}
We begin by writing
\begin{equation*}
	\begin{split}
		\mathbb{E}\Big[\sum_{z\in\mathcal{H}_{\alpha}}f\big(|z|_{\mathrm{h}}-R\big)\Big]&=\sum_{n=0}^{\infty}\mathbb{E}\Big[f\big(|\rho_n^{(\alpha)}|_{\mathrm{h}}-R\big)\Big]\\
		&=\sum_{n=0}^{\infty}\int_{0}^{+\infty}f\Big(\log{\frac{1+r}{1-r}}-R\Big)2k_n^{(\alpha)}r^{2n+1}(1-r^2)^{\alpha-1}\mathrm{d}r,
	\end{split}
\end{equation*}
where we recall the notation $k_n^{(\alpha)}=\frac{\Gamma(\alpha+n+1)}{\Gamma(\alpha)\Gamma(n+1)}$.
By taking $x=\log{\frac{1+r}{1-r}}-R$,
we get 
\[r=\frac{e^{R+x}-1}{e^{R+x}+1} \quad \mbox{and}
\quad \mathrm{d}r=\frac{2e^{R+x}}{(e^{R+x}+1)^2}\mathrm{d}x,\]
so that
\begin{equation*}
	\begin{split}
		&\quad\,\mathbb{E}\Big[\sum_{z\in\mathcal{H}_{\alpha}}f\big(|z|_{\mathrm{h}}-R\big)\Big]\\
		&=4\sum_{n=0}^{\infty}\int_{-R}^{+\infty}f(x)k_n^{(\alpha)}\Big(\frac{e^{R+x}-1}{e^{R+x}+1}\Big)^{2n+1}\Big[1-\Big(\frac{e^{R+x}-1}{e^{R+x}+1}\Big)^2\Big]^{\alpha-1}\frac{e^{R+x}}{(e^{R+x}+1)^2}\mathrm{d}x\\
		&=4^{\alpha}\int_{-R}^{+\infty}\sum_{n=0}^{\infty}f(x)k_n^{(\alpha)}\Big(1-\frac{2}{e^{R+x}+1}\Big)^{2n+1}\frac{e^{\alpha(R+x)}}{(e^{R+x}+1)^{2\alpha}}\mathrm{d}x	\\
				&=4^{\alpha}\int_{-R}^{+\infty}\sum_{n=0}^{\infty}f(x)k_n^{(\alpha)}\Big(1-\frac{2}{e^{R+x}+1}\Big)^{2n+1}
				\frac{e^{-\alpha(R+x)}}{(1+e^{-(R+x)})^{2\alpha}
				}\mathrm{d}x.
	\end{split}
\end{equation*}
Notice that we have interchanged the summation and integration here
which is possible since $f$
is bounded and compactly supported. Therefore, by taking $n=\lfloor te^{R+x}\rfloor$, we get
\begin{equation}\label{expected value one}
	\begin{split}
		&\quad\,e^{-R}\,\mathbb{E}\Big[\sum_{z\in\mathcal{H}_{\alpha}}f\big(|z|_{\mathrm{h}}-R\big)\Big]\\
		&=4^{\alpha}\int_{-R}^{+\infty}
		\left(\sum_{n=0}^{\infty}f(x) e^x k_n^{(\alpha)}\Big(1-\frac{2}{e^{R+x}+1}\Big)^{2n+1}
				\frac{e^{-\alpha(R+x)}}{(1+e^{-(R+x)})^{2\alpha}
				}
								(e^{-(R+x)}) \right)				
				\mathrm{d}x\\
		&=4^{\alpha}\int_{-R}^{+\infty}
		\left(\int_{0}^{+\infty}f(x) e^x k_{\lfloor te^{R+x}\rfloor}^{(\alpha)}\Big(1-\frac{2}{e^{R+x}+1}\Big)^{2{\lfloor te^{R+x}\rfloor}+1}\frac{e^{-\alpha(R+x)}}{(1+e^{-(R+x)})^{2\alpha}}
\mathrm{d}t \right)\mathrm{d}x.
	\end{split}
\end{equation}
We define the functions $\Theta_1$ and $\Theta_2$ of $(x,t,R)$ 
and the function $\Theta_3$ of $(x,R)$ by
\begin{equation}\label{estimation one}
	\begin{split}
		k_{\lfloor te^{R+x}\rfloor}^{(\alpha)}=\frac{t^{\alpha}e^{\alpha (R+x)}}{\Gamma(\alpha)}+\Theta_1(x,t,R),
	\end{split}
\end{equation}

\begin{equation}\label{estimation two}
	\begin{split}
		\Big(1-\frac{2}{e^{R+x}+1}\Big)^{2{\lfloor te^{R+x}\rfloor}+1}=e^{-4t}+\Theta_2(x,t,R),
	\end{split}
\end{equation}

\begin{equation}\label{estimation three}
	\begin{split}
\frac{1}{(1+e^{-(R+x)})^{2\alpha}}=1+\Theta_3(x,R).
	\end{split}
\end{equation}
By Lemma \ref{lem:Stirling} and Lemma \ref{lem:BoundExp}
(and by a standard fact for
\eqref{estimation three}),
there exists a constant $C>0$ such that, for $R$ large enough,

\begin{equation*}
	\begin{split}
		|\Theta_1(x,t,R)|\leq\left\{\begin{array}{ll}
			C,&(x,t)\in\supp f \times [0,e^{-R})\\
			Ct^{\alpha-1}e^{(\alpha-1)R},&(x,t)\in\supp f \times [e^{-R},+\infty)
		\end{array}\right.,
	\end{split}
\end{equation*}
\begin{equation*}
	\begin{split}
		|\Theta_2(x,t,R)|\leq C e^{-t/C}e^{-R}, \quad
		(x,t) \in \supp f \times [0,+\infty),
	\end{split}
\end{equation*}
\begin{equation*}
	\begin{split}
		|\Theta_3(x,R)|\leq Ce^{-R}, \quad 
		x \in \supp f.
	\end{split}
\end{equation*}

With the help of the estimates \eqref{estimation one}, \eqref{estimation two} and \eqref{estimation three},  it follows from \eqref{expected value one} that
\begin{equation*}
	\begin{split}
		&\quad\,e^{-R}\,\mathbb{E}\Big[\sum_{z\in\mathcal{H}_{\alpha}}f\big(|z|_{\mathrm{h}}-R\big)\Big]\\
		&=4^{\alpha}\int_{-R}^{+\infty}\int_{0}^{+\infty}f(x)e^x\frac{t^{\alpha}e^{\alpha (R+x)}}{\Gamma(\alpha)}e^{-4t}e^{-\alpha(R+x)}\mathrm{d}t\mathrm{d}x+O(e^{-R})\\
		&=\frac{4^{\alpha}}{\Gamma(\alpha)}\int_{0}^{+\infty}t^{\alpha}e^{-4t}\mathrm{d}t\int_{\mathbb{R}}f(x)e^x\mathrm{d}x+O(e^{-R})\\
		&=\frac{\alpha}{4}\int_{\mathbb{R}}f(x)e^{x}\mathrm{d}x+O(e^{-R}).
	\end{split}
\end{equation*}
That is,
\begin{equation*}
	\begin{split}
		\mathbb{E}\Big[\sum_{z\in\mathcal{H}_{\alpha}}f\big(|z|_{\mathrm{h}}-R\big)\Big]=2C_R^{(\alpha)}\int_{\mathbb{R}}f(x)e^{x}\mathrm{d}x+O(1).
	\end{split}
\end{equation*}

\subsection{Variance calculation for Theorem~\ref{Extreme Hyperbolic Normality}}\label{Variance calculation for Extreme Hyperbolic Normality}
We now turn to calculate the variance of the linear statistics $\sum_{z\in\mathcal{H}_{\alpha}}f\big(|z|_{\mathrm{h}}-R\big)$,
\begin{equation*}
	\begin{split}
		\Var\Big(\sum_{z\in\mathcal{H}_{\alpha}}f\big(|z|_{\mathrm{h}}-R\big)\Big)&=\sum_{n=0}^{\infty}\Var\Big(f\big(|\rho_n^{(\alpha)}|_{\mathrm{h}}-R\big)\Big)\\
		&=\sum_{n=0}^{\infty}\mathbb{E}\Big[f^2\big(|\rho_n^{(\alpha)}|_{\mathrm{h}}-R\big)\Big]-\sum_{n=0}^{\infty}\mathbb{E}\Big[f\big(|\rho_n^{(\alpha)}|_{\mathrm{h}}-R\big)\Big]^2.
	\end{split}
\end{equation*}
Due to the asymptotics in
Lemma \ref{lem:BoundExp} and
Lemma \ref{lem:Stirling} we can obtain a bit more that it is
needed.
For the first term $\sum_{n=0}^{\infty}\mathbb{E}\big[f^2\big(|\rho_n^{(\alpha)}|_{\mathrm{h}}-R\big)\big]$, by Subsection~\ref{Expected value calculation for Extreme Hyperbolic Normality}, we have
\begin{equation*}
	\begin{split}
		\sum_{n=0}^{\infty}\mathbb{E}\Big[f^2\big(|\rho_n^{(\alpha)}|_{\mathrm{h}}-R\big)\Big]=2C_R^{(\alpha)}\int_{\mathbb{R}}f^2(x)e^{x}\mathrm{d}x+O(1).
	\end{split}
\end{equation*}
As for the second term $\sum_{n=0}^{\infty}\mathbb{E}\big[f\big(|\rho_n^{(\alpha)}|_{\mathrm{h}}-R\big)\big]^2$, we can study it in the same way,
\begin{equation*}
	\begin{split}
		&\quad\,\sum_{n=0}^{\infty}\mathbb{E}\Big[f\big(|\rho_n^{(\alpha)}|_{\mathrm{h}}-R\big)\Big]^2\\
		&=4^{2\alpha}\sum_{n=0}^{\infty}\left(\int_{-R}^{+\infty}f(x)k_n^{(\alpha)}\Big(1-\frac{2}{e^{R+x}+1}\Big)^{2n+1}\frac{e^{-\alpha(R+x)}}{(1+e^{-(R+x)})^{2\alpha}}\mathrm{d}x\right)^2.
	\end{split}
\end{equation*}
Notice that the summation is outside the integration. Setting $n=\lfloor te^R\rfloor$, we have
\begin{equation}\label{expected value two}
	\begin{split}
		&\quad\,e^{-R}\,\sum_{n=0}^{\infty}\mathbb{E}\Big[f\big(|\rho_n^{(\alpha)}|_{\mathrm{h}}-R\big)\Big]^2\\
				&=4^{2\alpha}\sum_{n=0}^{\infty}\left(\int_{-R}^{+\infty}f(x)k_n^{(\alpha)}\Big(1-\frac{2}{e^{R+x}+1}\Big)^{2n+1}\frac{e^{-\alpha(R+x)}}{(1+e^{-(R+x)})^{2\alpha}}\mathrm{d}x\right)^2 e^{-R}\\
		&=4^{2\alpha}\int_{0}^{+\infty}\left(\int_{-R}^{+\infty}f(x)k_{\lfloor te^{R}\rfloor}^{(\alpha)}\Big(1-\frac{2}{e^{R+x}+1}\Big)^{2\lfloor te^{R}\rfloor+1}\frac{e^{-\alpha(R+x)}}{(1+
		e^{-(R+x)})^{2\alpha}}\mathrm{d}x\right)^2\mathrm{d}t.
	\end{split}
\end{equation}
We define the function $\Theta_4$
of $(t,R)$ and the function $\Theta_5$
of $(x,t,R)$ by
\begin{equation}\label{estimation four}
	\begin{split}
		k_{\lfloor te^{R}\rfloor}^{(\alpha)}=\frac{t^{\alpha}e^{\alpha R}}{\Gamma(\alpha)}+\Theta_4(t,R),
	\end{split}
\end{equation}

\begin{equation}\label{estimation five}
	\begin{split}
		\Big(1-\frac{2}{e^{R+x}+1}\Big)^{2{\lfloor te^{R}\rfloor}+1}=e^{-4te^{-x}}+\Theta_5(x,t,R),
	\end{split}
\end{equation}
so that, by Lemma \ref{lem:Stirling}
and Lemma \ref{lem:BoundExp}, there exists
a constant $C>0$ 
 such that, for $R$ large enough, we have the bounds
\begin{equation*}
	\begin{split}
		|\Theta_4(t,R)|\leq\left\{\begin{array}{ll}
			C,&t\in[0,e^{-R})\\
			C t^{\alpha-1}e^{(\alpha-1)R},&t\in[e^{-R},+\infty)
		\end{array}\right.,
	\end{split}
\end{equation*}
\begin{equation*}
	\begin{split}
		|\Theta_5(x,t,R)|\leq C e^{-t/C}e^{-R}
		\quad \mbox{ whenever } (x,t)\in\supp f \times [0,+\infty).
	\end{split}
\end{equation*}
With the help of these estimates
and recalling the standard estimate \eqref{estimation three}, it follows  that
\begin{equation*}
	\begin{split}
		&\quad\,e^{-R}\, \sum_{n=0}^{\infty}\mathbb{E}\Big[f\big(|\rho_n^{(\alpha)}|_{\mathrm{h}}-R\big)\Big]^2\\
		&=4^{2\alpha}\int_{0}^{+\infty}\left(\int_{-R}^{+\infty}f(x)\frac{t^{\alpha}e^{\alpha R}}{\Gamma(\alpha)}e^{-4te^{-x}}
		e^{-\alpha(R+x)}\mathrm{d}x\right)^2\mathrm{d}t+O(e^{-R})\\
		&=\frac{4^{2\alpha}}{\Gamma^2(\alpha)}\int_{\mathbb{R}^2}f(x)f(y)e^{-\alpha (x+y)}\int_{0}^{+\infty}t^{2\alpha}e^{-4t(e^{-x}+e^{-y})}\mathrm{d}t\mathrm{d}x\mathrm{d}y+O(e^{-R})\\
		&=\frac{1}{4\Gamma^2(\alpha)}\int_{\mathbb{R}^2}f(x)f(y)\frac{e^{-\alpha (x+y)}}{(e^{-x}+e^{-y})^{2\alpha+1}}\int_{0}^{+\infty}u^{2\alpha}e^{-u}\mathrm{d}u\mathrm{d}x\mathrm{d}y+O(e^{-R})\\
		&=\frac{\alpha}{4B(\alpha,\alpha+1)}\int_{\mathbb{R}^2}f(x)f(y)\frac{e^{(\alpha+1)(x+y)}}{\left(e^{x}+e^{y}\right)^{2\alpha+1}}\mathrm{d}x\mathrm{d}y+O(e^{-R}),
	\end{split}
\end{equation*}
that is,
\begin{equation*}
	\begin{split}
		\sum_{n=0}^{\infty}\mathbb{E}\Big[f\big(|\rho_n^{(\alpha)}|_{\mathrm{h}}-R\big)\Big]^2=\frac{2C_R^{(\alpha)}}{B(\alpha,\alpha+1)}\int_{\mathbb{R}^2}f(x)f(y)\frac{e^{(\alpha+1)(x+y)}}{\left(e^{x}+e^{y}\right)^{2\alpha+1}}\mathrm{d}x\mathrm{d}y+O(1).
	\end{split}
\end{equation*}
Therefore, we obtain
\begin{equation*}
	\begin{split}
		&\quad\,\Var\Big(\sum_{z\in\mathcal{H}_{\alpha}}f\big(|z|_{\mathrm{h}}-R\big)\Big)\\
		&=2C_R^{(\alpha)}\int_{\mathbb{R}}f^2(x)e^{x}\mathrm{d}x-\frac{2C_R^{(\alpha)}}{B(\alpha,\alpha+1)}\int_{\mathbb{R}^2}f(x)f(y)\frac{e^{(\alpha+1)(x+y)}}{\left(e^{x}+e^{y}\right)^{2\alpha+1}}\mathrm{d}x\mathrm{d}y+O(1).
	\end{split}
\end{equation*}
By noticing that
\begin{equation*}
	\begin{split}
		\frac{1}{B(\alpha,\alpha+1)}\int_{\mathbb{R}}\frac{e^{(\alpha+1)(x+y)}}{\left(e^{x}+e^{y}\right)^{2\alpha+1}}\mathrm{d}y=e^x,
	\end{split}
\end{equation*}
we may conclude that
\begin{equation*}
	\begin{split}
		\Var\Big(\sum_{z\in\mathcal{H}_{\alpha}}f\big(|z|_{\mathrm{h}}-R\big)\Big)=\frac{C_R^{(\alpha)}}{B(\alpha,\alpha+1)}\int_{\mathbb{R}^2}\big[f(x)-f(y)\big]^2\frac{e^{(\alpha+1)(x+y)}}{\left(e^{x}+e^{y}\right)^{2\alpha+1}}\mathrm{d}x\mathrm{d}y+O(1).
	\end{split}
\end{equation*}

\subsection{Soshnikov's conditions and conclusion of the proof}
We will use the following central limit theorem, derived by Soshnikov \cite[Theorem 1]{Sos}, to prove Theorem~\ref{Extreme Hyperbolic Normality}.

\begin{theorem}[The central limit theorem of Soshnikov]\label{Soshnikov}
	For $L>0$, suppose that $\mathscr{X}_L$ is a determinantal point process on a locally compact Polish space $X_L$ with Hermitian kernel $K_L$ with respect to the reference Radon measure $\mu_L$. By slightly abusing the notation, also denote $K_L$ the associated integral operator with integral kernel $K_L$ on $L^2(X_L,\mu_L)$, that is, $K_L:L^2(X_L,\mu_L)\to L^2(X_L,\mu_L)$ is defined by 
	\begin{equation*}
		\begin{split}
			K_Lf(x)=\int_{X_L}K_L(x,y)f(y)d\mu_L(y),\quad f\in L^2(X_L,\mu_L),
		\end{split}
	\end{equation*}
	and suppose that for any pre-compact Borel set $\Delta\subset X_L$, the operator $K_L\, \chi_{\Delta}$ is trace-class, where $\chi_{\Delta}$ denotes the multiplication operator by the indicator of $\Delta$.
	
	Let $f_L$ be a real-valued bounded measurable function on $X_L$ with compact support and consider the linear statistics $S_{f_L}=\sum_{x\in\mathscr{X}_L}f_L(x)$. If 
\begin{itemize}	
\item	$\Var S_{f_L}\to +\infty$ as $L\to +\infty$, 
\item $\sup_{x\in X_L}|f_L(x)|=o((\Var{f_L})^\varepsilon)$ as $L\to +\infty$ for any $\varepsilon>0$ and 
\item $\mathbb{E}S_{|f_L|}=O((\Var S_{f_L})^\delta)$ as $L\to +\infty$ for some $\delta>0$,
\end{itemize} 
then the centered normalized linear statistics converges in law to
the standard normal distribution, i.e.,
	\begin{equation*}
		\begin{split}
			\frac{S_{f_L}-\mathbb{E}S_{f_L}}{\sqrt{\Var S_{f_L}}}
			\xrightarrow[L \to \infty]{\mathrm{law}} \mathcal N(0,1)
		\end{split}
	\end{equation*} 
\end{theorem}

When $R\to+\infty$, with the help of Subsection~\ref{Expected value calculation for Extreme Hyperbolic Normality} and Subsection~\ref{Variance calculation for Extreme Hyperbolic Normality}, we have
\begin{equation*}
	\begin{split}
		\mathbb{E}\Big[\sum_{z\in\mathcal{H}_{\alpha}}\big|f\big(|z|_{\mathrm{h}}-R\big)\big|\Big]\sim2C_R^{(\alpha)}\int_{\mathbb{R}}|f(x)|e^{x}\mathrm{d}x,
		\quad \mbox{ and}
	\end{split}
\end{equation*}
\begin{equation*}
	\begin{split}
		\Var\Big(\sum_{z\in\mathcal{H}_{\alpha}}f\big(|z|_{\mathrm{h}}-R\big)\Big)\sim C_R^{(\alpha)}V_f^{(\alpha)}.
	\end{split}
\end{equation*}
Applying Soshnikov's Theorem~\ref{Soshnikov} to $\sum_{z\in\mathcal{H}_{\alpha}}f\big(|z|_{\mathrm{h}}-R\big)$, we get
\begin{equation*}
	\begin{split}
		\frac{\sum_{z\in\mathcal{H}_{\alpha}}f\big(|z|_{\mathrm{h}}-R\big)-\mathbb{E}\big[\sum_{z\in\mathcal{H}_{\alpha}}f\big(|z|_{\mathrm{h}}-R\big)\big]}{\sqrt{\Var\big(\sum_{z\in\mathcal{H}_{\alpha}}f\big(|z|_{\mathrm{h}}-R\big)\big) }}\xrightarrow[R\to+\infty]{\mathrm{law}}\mathcal{N}(0,1).
	\end{split}
\end{equation*}
This gives that
\begin{equation*}
	\begin{split}
		\frac{1}{\sqrt{C_R^{(\alpha)}}}\sum_{z\in\mathcal{H}_{\alpha}}f\big(|z|_{\mathrm{h}}-R\big)-\frac{1}{\sqrt{C_R^{(\alpha)}}}\mathbb{E}\Big[\sum_{z\in\mathcal{H}_{\alpha}}f\big(|z|_{\mathrm{h}}-R\big)\Big]\xrightarrow[R\to+\infty]{\mathrm{law}}\mathcal{N}\big(0,V_f^{(\alpha)}\big).
	\end{split}
\end{equation*}
Moreover, it follows from Subsection~\ref{Expected value calculation for Extreme Hyperbolic Normality} that
\begin{equation*}
	\begin{split}
		\frac{1}{\sqrt{C_R^{(\alpha)}}}\mathbb{E}\Big[\sum_{z\in\mathcal{H}_{\alpha}}f\big(|z|_{\mathrm{h}}-R\big)\Big]-2\sqrt{C_R^{(\alpha)}}\int_{\mathbb{R}}f(x)e^{x}\mathrm{d}x\xrightarrow[R\to+\infty]{}0,
	\end{split}
\end{equation*}
hence
\begin{equation*}
	\begin{split}
		\frac{1}{\sqrt{C_R^{(\alpha)}}}\sum_{z\in\mathcal{H}_{\alpha}}f\big(|z|_{\mathrm{h}}-R\big)-2\sqrt{C_R^{(\alpha)}}\int_{\mathbb{R}}f(x)e^{x}\mathrm{d}x\xrightarrow[R\to+\infty]{\mathrm{law}}\mathcal{N}\big(0,V_f^{(\alpha)}\big).
	\end{split}
\end{equation*}
This completes the proof of Theorem~\ref{Extreme Hyperbolic Normality}.

\section{The Ginibre process proof of Theorem~\ref{Extreme Ginibre Normality}}
The proof of Theorem \ref{Extreme Ginibre Normality} follows the same ideas as the proof of Theorem \ref{Extreme Hyperbolic Normality}. 
Moreover, the calculation of the variance can be nicely done
using the same
Riemann sum's argument as before.
Nevertheless, 
we have decided to use 
formulas \eqref{Expected value} and \eqref{Variance}
to emphasize a slightly different way.

\subsection{Expected value calculation for Theorem~\ref{Extreme Ginibre Normality}}\label{Expected value calculation for Extreme Ginibre Normality}
For the expectation of the random variable $\sum_{z\in\mathcal{G}}f\big(|z|-R\big)$,
notice that, since $
		K_{\mathcal{G}}(z,z)=1/\pi$,
we have
\begin{equation*}
	\begin{split}
		\mathbb{E}\Big[\sum_{z\in\mathcal{G}}f\big(|z|-R\big)\Big]&=\int_{\mathbb{C}}f\big(|z|-R\big)K_{\mathcal{G}}(z,z)\mathrm{d}m(z)\\
		&=\frac{1}{\pi}\int_{\mathbb{C}}f\big(|z|-R\big)\mathrm{d}m(z).
	\end{split}
\end{equation*}

By using polar coordinates and making a variable substitution, we get
\begin{equation*}
	\begin{split}
		\mathbb{E}\Big[\sum_{z\in\mathcal{G}}f\big(|z|-R\big)\Big]&=2\int_{0}^{+\infty}f(r-R)r\mathrm{d}r\\
		&=2\int_{-R}^{+\infty}f(x)\mathrm{d}x\cdot R+2\int_{-R}^{+\infty}xf(x)\mathrm{d}x,
	\end{split}
\end{equation*}
that is,
\begin{equation*}
	\begin{split}
		\mathbb{E}\Big[\sum_{z\in\mathcal{G}}f\big(|z|-R\big)\Big]=2R\int_{\mathbb{R}}f(x)\mathrm{d}x+O(1).
	\end{split}
\end{equation*}

\subsection{Variance calculation for Theorem~\ref{Extreme Ginibre Normality}}\label{Variance calculation for Extreme Ginibre Normality}
For the variance of the random variable $\sum_{z\in\mathcal{G}}f\big(|z|-R\big)$,
we use that
\begin{equation*}
	\begin{split}
		\big|K_{\mathcal{G}}(z,w)\big|^2=\frac{1}{\pi^2}e^{-|z-w|^2},
	\end{split}
\end{equation*}
to obtain
\begin{equation*}
	\begin{split}
		\Var\Big(\sum_{z\in\mathcal{G}}f\big(|z|-R\big)\Big)&=\frac{1}{2}\int_{\mathbb{C}^2}\big[f\big(|z|-R\big)-f\big(|w|-R\big)\big]^2\big|K_{\mathcal{G}}(z,w)\big|^2\mathrm{d}m(z)\mathrm{d}m(w)\\
		&=\frac{1}{2\pi^2}\int_{\mathbb{C}^2}\big[f\big(|z|-R\big)-f\big(|w|-R\big)\big]^2e^{-|z-w|^2}\mathrm{d}m(z)\mathrm{d}m(w).
	\end{split}
\end{equation*}
By using polar coordinates, we get
\begin{equation*}
	\begin{split}
		&\quad\,\Var\Big(\sum_{z\in\mathcal{G}}f\big(|z|-R\big)\Big)\\
		&=\frac{1}{2\pi^2}\int_{0}^{+\infty}\int_{0}^{+\infty}\int_{-\pi}^{\pi}\int_{-\pi}^{\pi}\big[f(r-R)-f(\rho-R)\big]^2e^{-|re^{i\theta}-\rho e^{i\varphi}|^2}r\rho\mathrm{d}\theta\mathrm{d}\varphi\mathrm{d}r\mathrm{d}\rho\\
		&=\frac{1}{\pi}\int_{0}^{+\infty}\int_{0}^{+\infty}\int_{-\pi}^{\pi}\big[f(r-R)-f(\rho-R)\big]^2e^{-|re^{i\theta}-\rho|^2}r\rho\mathrm{d}\theta\mathrm{d}\varphi\mathrm{d}r\\
		&=\frac{2}{\pi}\int_{0}^{+\infty}\int_{0}^{+\infty}\big[f(r-R)-f(\rho-R)\big]^2e^{-r^2-\rho^2}r\rho\int_{0}^{\pi}e^{2r\rho\cos\theta}\mathrm{d}\theta\mathrm{d}r\mathrm{d}\rho.
	\end{split}
\end{equation*}
By the change of variables $x=r-R$ and $y=\rho-R$, we have
\begin{equation*}
	\begin{split}
		&\quad\,\Var\Big(\sum_{z\in\mathcal{G}}f\big(|z|-R\big)\Big)\\
		&=\frac{2}{\pi}\int_{-R}^{+\infty}\int_{-R}^{+\infty}\big[f(x)-f(y)\big]^2e^{-(R+x)^2-(R+y)^2}(R+x)(R+y)\int_{0}^{\pi}e^{2(R+x)(R+y)\cos\theta}\mathrm{d}\theta\mathrm{d}x\mathrm{d}y\\
		&=\frac{2}{\pi}\int_{-R}^{+\infty}\int_{-R}^{+\infty}\big[f(x)-f(y)\big]^2e^{-(x-y)^2}(R+x)(R+y)\int_{0}^{\pi}e^{-4(R+x)(R+y)\sin^2\frac{\theta}{2}}\mathrm{d}\theta\mathrm{d}x\mathrm{d}y.
	\end{split}
\end{equation*}
Let $\theta=t/\sqrt{(R+x)(R+y)}$, we obtain
\begin{equation*}
	\begin{split}
		\Var\Big(\sum_{z\in\mathcal{G}}f\big(|z|-R\big)\Big)&=\frac{2}{\pi}\int_{-R}^{+\infty}\int_{-R}^{+\infty}\big[f(x)-f(y)\big]^2e^{-(x-y)^2}\sqrt{(R+x)(R+y)}\\
		&\qquad\qquad\qquad\cdot\int_{0}^{\pi\sqrt{(R+x)(R+y)}}e^{-4(R+x)(R+y)\sin^2\frac{t}{2\sqrt{(R+x)(R+y)}}}\mathrm{d}t\mathrm{d}x\mathrm{d}y.
	\end{split}
\end{equation*}
Consider the integral
\begin{equation*}
	\begin{split}
		&\int_{\mathbb{R}^3}\big[f(x)-f(y)\big]^2e^{-(x-y)^2}\sqrt{(1+x/R)(1+y/R)}\,e^{-4(R+x)(R+y)\sin^2\frac{t}{2\sqrt{(R+x)(R+y)}}}\\
		&\qquad\qquad\qquad\qquad\qquad\,\,\cdot\chi_{\left[0,\pi\sqrt{(R+x)(R+y)}\right]}(t)\chi_{[-R,+\infty)}(x)\chi_{[-R,+\infty)}(y)\mathrm{d}t\mathrm{d}x\mathrm{d}y.
	\end{split}
\end{equation*}
When $R$ is sufficiently large, the integrand is dominated by
\begin{equation*}
	\begin{split}
		\big[f(x)-f(y)\big]^2e^{-(x-y)^2}\sqrt{(1+|x|)(1+|y|)}\,e^{-\frac{4}{\pi^2}t^2}\chi_{[0,+\infty]}(t),
	\end{split}
\end{equation*}
hence by the dominated convergence theorem, as $R\to+\infty$, the integral converges to
\begin{equation*}
	\begin{split}
		\frac{\sqrt{\pi}}{2}\int_{\mathbb{R}^2}\big[f(x)-f(y)\big]^2e^{-(x-y)^2}\mathrm{d}x\mathrm{d}y.
	\end{split}
\end{equation*}
It follows that, when $R\to+\infty$,
\begin{equation*}
	\begin{split}
		\Var\Big(\sum_{z\in\mathcal{G}}f\big(|z|-R\big)\Big)\sim\frac{R}{\sqrt{\pi}}\int_{\mathbb{R}^2}\big[f(x)-f(y)\big]^2e^{-(x-y)^2}\mathrm{d}x\mathrm{d}y.
	\end{split}
\end{equation*}

\subsection{Soshnikov's conditions and conclusion of the proof}
When $R\to+\infty$, with the help of Subsection~\ref{Expected value calculation for Extreme Ginibre Normality} and Subsection~\ref{Variance calculation for Extreme Ginibre Normality}, we have
\begin{equation*}
	\begin{split}
		\mathbb{E}\Big[\sum_{z\in\mathcal{G}}\big|f\big(|z|-R\big)\big|\Big]\sim2R\int_{\mathbb{R}}|f(x)|\mathrm{d}x,\quad \mbox{ and }
	\end{split}
\end{equation*}
\begin{equation*}
	\begin{split}
		\Var\Big(\sum_{z\in\mathcal{G}}f\big(|z|-R\big)\Big)\sim RV_f.
	\end{split}
\end{equation*}
Applying Soshnikov's Theorem~\ref{Soshnikov} to $\sum_{z\in\mathcal{G}}f\big(|z|-R\big)$, we get
\begin{equation*}
	\begin{split}
		\frac{\sum_{z\in\mathcal{G}}f\big(|z|-R\big)-\mathbb{E}\big[\sum_{z\in\mathcal{G}}f\big(|z|-R\big)\big]}{\sqrt{\Var\big(\sum_{z\in\mathcal{G}}f\big(|z|-R\big)\big)}}\xrightarrow[R\to+\infty]{\mathrm{law}}\mathcal{N}(0,1).
	\end{split}
\end{equation*}
This gives that
\begin{equation*}
	\begin{split}
		\frac{1}{\sqrt{R}}\sum_{z\in\mathcal{G}}f\big(|z|-R\big)-\frac{1}{\sqrt{R}}\mathbb{E}\Big[\sum_{z\in\mathcal{G}}f\big(|z|-R\big)\Big]\xrightarrow[R\to+\infty]{\mathrm{law}}\mathcal{N}(0,V_f).
	\end{split}
\end{equation*}
Moreover, it follows from Subsection~\ref{Expected value calculation for Extreme Ginibre Normality} that
\begin{equation*}
	\begin{split}
		\frac{1}{\sqrt{R}}\mathbb{E}\Big[\sum_{z\in\mathcal{G}}f\big(|z|-R\big)\Big]-2\sqrt{R}\int_{\mathbb{R}}f(x)e^{x}\mathrm{d}x\xrightarrow[R\to+\infty]{}0,
	\end{split}
\end{equation*}
hence
\begin{equation*}
	\begin{split}
		\frac{1}{\sqrt{R}}\sum_{z\in\mathcal{G}}f\big(|z|-R\big)-2\sqrt{R}\int_{\mathbb{R}}f(x)e^{x}\mathrm{d}x\xrightarrow[R\to+\infty]{\mathrm{law}}\mathcal{N}(0,V_f).
	\end{split}
\end{equation*}
This completes the proof of Theorem~\ref{Extreme Ginibre Normality}.

\section{Proof of Theorem~\ref{White Noise Hyperbolic} (hyperbolic case)}\label{Proof of White Noise Hyperbolic}
In the case $1\ll a_R\ll e^R$ as $R\to+\infty$, the proof of Theorem~\ref{White Noise Hyperbolic} follows the same steps as the proof of Theorem~\ref{Extreme Hyperbolic Normality}: first we need to understand the expected value, then to calculate the limiting variance, and finally to use Soshnikov's Theorem~\ref{Soshnikov}.

Let us begin by writing
\begin{equation*}
	\begin{split}
		&\quad\,\mathbb{E}\Big[\sum_{z\in\mathcal{H}_{\alpha}}f\big(a_R(|z|_{\mathrm{h}}-R)\big)\Big]\\
		&=4^{\alpha}\int_{-R}^{+\infty}\sum_{n=0}^{\infty}f(a_R y)k_n^{(\alpha)}\Big(1-\frac{2}{e^{R+y}+1}\Big)^{2n+1}\frac{e^{-\alpha(R+y)}}{(1+e^{-(R+y)})^{2\alpha}}\mathrm{d}y\\
		&=4^{\alpha}\int_{-Ra_R}^{+\infty}\sum_{n=0}^{\infty}f(x)k_n^{(\alpha)}\Big(1-\frac{2}{e^{R+\frac{x}{a_R}}+1}\Big)^{2n+1}\frac{e^{-\alpha(R+\frac{x}{a_R})}}{\big(1+e^{-(R+\frac{x}{a_R})}\big)^{2\alpha}}
		\frac{\mathrm{d}x}{a_R}.
	\end{split}
\end{equation*}
By taking $n=\lfloor te^{R+x/a_R}\rfloor$, we get
\begin{equation*}
	\begin{split}
		&\quad\,\frac{a_R}{e^R}\,\mathbb{E}\Big[\sum_{z\in\mathcal{H}_{\alpha}}f\big(a_R(|z|_{\mathrm{h}}-R)\big)\Big]\\
		&=4^{\alpha}\int_{-Ra_R}^{+\infty}
		\left(\sum_{n=0}^{\infty}f(x)
		e^{\frac{x}{a_R}} k_n^{(\alpha)}\Big(1-\frac{2}{e^{R+\frac{x}{a_R}}+1}\Big)^{2n+1}\frac{e^{-\alpha(R+\frac{x}{a_R})}}{\big(1+e^{-(R+\frac{x}{a_R})}\big)^{2\alpha}}
		 (		e^{-(R+\frac{x}{a_R})}) \right)
		\mathrm{d}x\\
		&=4^{\alpha}\int_{-Ra_R}^{+\infty}
		\left(\int_{0}^{+\infty}f(x)e^{\frac{x}{a_R}}k_{\lfloor te^{R+\frac{x}{a_R}}\rfloor}^{(\alpha)}\Big(1-\frac{2}{e^{R+\frac{x}{a_R}}+1}\Big)^{2{\lfloor te^{R+\frac{x}{a_R}}\rfloor}+1}\frac{e^{-\alpha(R+\frac{x}{a_R})}}{(1+e^{-(R+\frac{x}{a_R})})^{2\alpha}}\mathrm{d}t
		\right)\mathrm{d}x.
	\end{split}
\end{equation*}
This can be done in the same way as in Subsection~\ref{Expected value calculation for Extreme Hyperbolic Normality}. Notice that $x\in\supp f$ and $1\ll a_R\ll e^R$ as $R\to+\infty$. We can show that for sufficiently large $R$,
\begin{equation*}
	\begin{split}
		&\quad\,\frac{a_R}{e^R}\mathbb{E}\Big[\sum_{z\in\mathcal{H}_{\alpha}}f\big(a_R(|z|_{\mathrm{h}}-R)\big)\Big]\\
		&=4^{\alpha}\int_{-Ra_R}^{+\infty}\int_{0}^{+\infty}f(x)\frac{t^{\alpha}e^{\alpha (R+\frac{x}{a_R})}}{\Gamma(\alpha)}e^{-4t}e^{-\alpha(R+\frac{x}{a_R})}
		\mathrm{d}t\mathrm{d}x+O(e^{-R})\\
		&=\frac{4^{\alpha}}{\Gamma(\alpha)}\int_{0}^{+\infty}t^{\alpha}e^{-4t}\mathrm{d}t\int_{\mathbb{R}}f(x)\mathrm{d}x+O(e^{-R})\\
		&=\frac{\alpha}{4}\int_{\mathbb{R}}f(x)\mathrm{d}x+O(e^{-R}),
	\end{split}
\end{equation*}
that is,
\begin{equation*}
	\begin{split}
		\mathbb{E}\Big[\sum_{z\in\mathcal{H}_{\alpha}}f\big(a_R(|z|_{\mathrm{h}}-R)\big)\Big]=\frac{2C_R^{(\alpha)}}{a_R}\int_{\mathbb{R}}f(x)\mathrm{d}x+O\Big(\frac{1}{a_R}\Big).
	\end{split}
\end{equation*}

As for the variance
\begin{equation*}
	\begin{split}
		\Var\Big(\sum_{z\in\mathcal{H}_{\alpha}}f\big(a_R(|z|_{\mathrm{h}}-R)\big)\Big)=\sum_{n=0}^{\infty}\mathbb{E}\Big[f^2\big(a_R(|\rho_n^{(\alpha)}|_{\mathrm{h}}-R)\big)\Big]-\sum_{n=0}^{\infty}\mathbb{E}\Big[f\big(a_R(|\rho_n^{(\alpha)}|_{\mathrm{h}}-R)\big)\Big]^2,
	\end{split}
\end{equation*}
the first term satisfies
\begin{equation*}
	\begin{split}
		\sum_{n=0}^{\infty}\mathbb{E}\Big[f^2\big(a_R(|\rho_n^{(\alpha)}|_{\mathrm{h}}-R)\big)\Big]=\frac{2C_R^{(\alpha)}}{a_R}\int_{\mathbb{R}}f^2(x)\mathrm{d}x+O\Big(\frac{1}{a_R}\Big).
	\end{split}
\end{equation*}
The second term can be dealt with in the same way as in Subsection~\ref{Variance calculation for Extreme Hyperbolic Normality},
\begin{equation*}
	\begin{split}
	&\quad\,\sum_{n=0}^{\infty}\mathbb{E}\Big[f\big(a_R(|\rho_n^{(\alpha)}|_{\mathrm{h}}-R)\big)\Big]^2\\
	&=4^{2\alpha}\sum_{n=0}^{\infty}\left(\int_{-R}^{+\infty}f(a_Rx)k_n^{(\alpha)}\Big(1-\frac{2}{e^{R+x}+1}\Big)^{2n+1}\frac{e^{\alpha(R+x)}}{(e^{R+x}+1)^{2\alpha}}\mathrm{d}x\right)^2\\
	&=4^{2\alpha}\sum_{n=0}^{\infty}\left(\int_{-Ra_R}^{+\infty}f(x)k_n^{(\alpha)}\Big(1-\frac{2}{e^{R+\frac{x}{a_R}}+1}\Big)^{2n+1}\frac{e^{\alpha(R+\frac{x}{a_R})}}{(e^{R+\frac{x}{a_R}}+1)^{2\alpha}}\mathrm{d}x\right)^2\cdot\frac{1}{a_R^2},
\end{split}
\end{equation*}
set $n=\lfloor te^R\rfloor$, then
\begin{equation*}
	\begin{split}
		&\quad\,\sum_{n=0}^{\infty}\mathbb{E}\Big[f\big(a_R(|\rho_n^{(\alpha)}|_{\mathrm{h}}-R)\big)\Big]^2\\
		&=4^{2\alpha}\int_{0}^{+\infty}\left(\int_{-Ra_R}^{+\infty}f(x)k_{\lfloor te^{R}\rfloor}^{(\alpha)}\Big(1-\frac{2}{e^{R+\frac{x}{a_R}}+1}\Big)^{2\lfloor te^{R}\rfloor+1}\frac{e^{\alpha(R+\frac{x}{a_R})}}{(e^{R+\frac{x}{a_R}}+1)^{2\alpha}}\mathrm{d}x\right)^2\mathrm{d}t\cdot\frac{e^R}{a_R^2}.
	\end{split}
\end{equation*}
Hence when $R\to+\infty$,
\begin{equation*}
	\begin{split}
		\sum_{n=0}^{\infty}\mathbb{E}\Big[f\big(a_R(|\rho_n^{(\alpha)}|_{\mathrm{h}}-R)\big)\Big]^2&\sim4^{2\alpha}\int_{0}^{+\infty}\left(\int_{-Ra_R}^{+\infty}f(x)\frac{t^{\alpha}e^{\alpha R}}{\Gamma(\alpha)}e^{-4t}e^{-2\alpha R}e^{\alpha R}\mathrm{d}x\right)^2\mathrm{d}t\cdot\frac{e^R}{a_R^2}\\
		&=\frac{4^{2\alpha}}{\Gamma^2(\alpha)}\int_{0}^{+\infty}t^{2\alpha}e^{-8t}\mathrm{d}t\left(\int_{\mathbb{R}}f(x)\mathrm{d}x\right)^2\cdot\frac{e^R}{a_R^2}\\
		&=\frac{\alpha}{2^{2\alpha+3}B(\alpha,\alpha+1)}\left(\int_{\mathbb{R}}f(x)\mathrm{d}x\right)^2\cdot\frac{e^R}{a_R^2}\\
		&=O\Big(\frac{C_R^{(\alpha)}}{a_R^2}\Big).
	\end{split}
\end{equation*}
Therefore, we obtain
\begin{equation*}
	\begin{split}
		\Var\Big(\sum_{z\in\mathcal{H}_{\alpha}}f\big(a_R(|z|_{\mathrm{h}}-R)\big)\Big)=\frac{2C_R^{(\alpha)}}{a_R}\int_{\mathbb{R}}f^2(x)\mathrm{d}x+O\Big(\frac{1}{a_R}\Big)+O\Big(\frac{C_R^{(\alpha)}}{a_R^2}\Big).
	\end{split}
\end{equation*}

In the case $1\ll a_R\ll e^R$ as $R\to+\infty$, applying Soshnikov's Theorem~\ref{Soshnikov}, the above calculations yield that
\begin{equation*}
	\begin{split}
		\frac{\sum_{z\in\mathcal{H}_{\alpha}}f\big(a_R(|z|_{\mathrm{h}}-R)\big)-\mathbb{E}\big[\sum_{z\in\mathcal{H}_{\alpha}}f\big(a_R(|z|_{\mathrm{h}}-R)\big)\big]}{\sqrt{\Var\big(\sum_{z\in\mathcal{H}_{\alpha}}f\big(a_R(|z|_{\mathrm{h}}-R)\big)\big) }}\xrightarrow[R\to+\infty]{\mathrm{law}}\mathcal{N}(0,1).
	\end{split}
\end{equation*}
This gives that
\begin{equation*}
	\begin{split}
		&\quad\,\sqrt{\frac{a_R}{C_R^{(\alpha)}}}\sum_{z\in\mathcal{H}_{\alpha}}f\big(a_R(|z|_{\mathrm{h}}-R)\big)-\sqrt{\frac{a_R}{C_R^{(\alpha)}}}\mathbb{E}\Big[\sum_{z\in\mathcal{H}_{\alpha}}f\big(a_R(|z|_{\mathrm{h}}-R)\big)\Big]\\
		&\xrightarrow[R\to+\infty]{\mathrm{law}}\mathcal{N}\Big(0,2\int_{\mathbb{R}}f^2(x)\mathrm{d}x\Big).
	\end{split}
\end{equation*}
Moreover,
\begin{equation*}
	\begin{split}
		\sqrt{\frac{a_R}{C_R^{(\alpha)}}}\mathbb{E}\Big[\sum_{z\in\mathcal{H}_{\alpha}}f\big(a_R(|z|_{\mathrm{h}}-R)\big)\Big]-2\sqrt{\frac{C_R^{(\alpha)}}{a_R}}\int_{\mathbb{R}}f(x)\mathrm{d}x\xrightarrow[R\to+\infty]{}0,
	\end{split}
\end{equation*}
hence
\begin{equation*}
	\begin{split}
		\sqrt{\frac{a_R}{C_R^{(\alpha)}}}\sum_{z\in\mathcal{H}_{\alpha}}f\big(a_R(|z|_{\mathrm{h}}-R)\big)-2\sqrt{\frac{C_R^{(\alpha)}}{a_R}}\int_{\mathbb{R}}f(x)\mathrm{d}x\xrightarrow[R\to+\infty]{\mathrm{law}}\mathcal{N}\Big(0,2\int_{\mathbb{R}}f^2(x)\mathrm{d}x\Big),
	\end{split}
\end{equation*}
This completes the proof of Theorem~\ref{White Noise Hyperbolic}.

\section{Proof of Theorem~\ref{White Noise Ginibre} (Ginibre case)}\label{Proof of White Noise Ginibre}
In the case $1\ll a_R\ll R$ as $R\to+\infty$, the proof of Theorem~\ref{White Noise Ginibre} follows the same steps as the proof of Theorem~\ref{Extreme Ginibre Normality}.

For the expectation of the random variable $\sum_{z\in\mathcal{G}}f\big(a_R(|z|-R)\big)$, we have
\begin{equation*}
	\begin{split}
		\mathbb{E}\Big[\sum_{z\in\mathcal{G}}f\big(a_R(|z|-R)\big)\Big]&=2\int_{-R}^{+\infty}f(a_Rx)\mathrm{d}x\cdot R+2\int_{-R}^{+\infty}xf(a_Rx)\mathrm{d}x\\
		&=2\int_{-Ra_R}^{+\infty}f(x)\mathrm{d}x\cdot\frac{R}{a_R}+2\int_{-Ra_R}^{+\infty}xf(x)\mathrm{d}x\cdot\frac{1}{a_R^2},
	\end{split}
\end{equation*}
that is,
\begin{equation*}
	\begin{split}
		\mathbb{E}\Big[\sum_{z\in\mathcal{G}}f\big(a_R(|z|-R)\big)\Big]=\frac{2R}{a_R}\int_{\mathbb{R}}f(x)\mathrm{d}x+O\Big(\frac{1}{a_R^2}\Big).
	\end{split}
\end{equation*}

As for the variance of the random variable $\sum_{z\in\mathcal{G}}f\big(a_R(|z|-R)\big)$, we have
\begin{equation*}
	\begin{split}
		&\quad\,\Var\Big(\sum_{z\in\mathcal{G}}f\big(a_R(|z|-R)\big)\Big)\\
		&=\frac{2}{\pi}\int_{-R}^{+\infty}\int_{-R}^{+\infty}\big[f(a_Rx)-f(a_Ry)\big]^2e^{-(x-y)^2}\sqrt{(R+x)(R+y)}\\
		&\qquad\qquad\qquad\qquad\cdot\int_{0}^{\pi\sqrt{(R+x)(R+y)}}e^{-4(R+x)(R+y)\sin^2\frac{t}{2\sqrt{(R+x)(R+y)}}}\mathrm{d}t\mathrm{d}x\mathrm{d}y\\
		&=\frac{2}{\pi}\int_{-Ra_R}^{+\infty}\int_{-Ra_R}^{+\infty}\big[f(x)-f(y)\big]^2e^{-\frac{(x-y)^2}{a_R^2}}\sqrt{\Big(R+\frac{x}{a_R}\Big)\Big(R+\frac{y}{a_R}\Big)}\\
		&\qquad\qquad\qquad\qquad\cdot\int_{0}^{\pi\sqrt{(R+\frac{x}{a_R})(R+\frac{y}{a_R})}}e^{-4(R+\frac{x}{a_R})(R+\frac{y}{a_R})\sin^2\frac{t}{2\sqrt{(R+\frac{x}{a_R})(R+\frac{y}{a_R})}}}\mathrm{d}t\mathrm{d}x\mathrm{d}y\cdot\frac{1}{a_R^2}\\
		&=I_1(R)+I_2(R),
	\end{split}
\end{equation*}
where
\begin{equation*}
	\begin{split}
		I_1(R)&=\frac{4}{\pi}\int_{-Ra_R}^{+\infty}\int_{-Ra_R}^{+\infty}\int_{0}^{\pi\sqrt{(R+\frac{x}{a_R})(R+\frac{y}{a_R})}}f^2(x)e^{-\frac{(x-y)^2}{a_R^2}}\sqrt{\Big(1+\frac{x}{Ra_R}\Big)\Big(1+\frac{y}{Ra_R}\Big)}\\
		&\qquad\qquad\qquad\qquad\qquad\qquad\cdot e^{-4(R+\frac{x}{a_R})(R+\frac{y}{a_R})\sin^2{\frac{t}{2\sqrt{(R+\frac{x}{a_R})(R+\frac{y}{a_R})}}}}\mathrm{d}t\mathrm{d}x\mathrm{d}y\cdot\frac{R}{a_R^2}\\
		&=\frac{4}{\pi}\int_{-R}^{+\infty}\int_{-Ra_R}^{+\infty}\int_{0}^{\pi\sqrt{(R+\frac{x}{a_R})(R+u)}}f^2(x)e^{-(\frac{x}{a_R}-u)^2}\sqrt{\Big(1+\frac{x}{Ra_R}\Big)\Big(1+\frac{u}{R}\Big)}\\
		&\qquad\qquad\qquad\qquad\qquad\qquad\cdot e^{-4(R+\frac{x}{a_R})(R+u)\sin^2{\frac{t}{2\sqrt{(R+\frac{x}{a_R})(R+u)}}}}\mathrm{d}t\mathrm{d}x\mathrm{d}u\cdot\frac{R}{a_R},
	\end{split}
\end{equation*}
and
\begin{equation*}
	\begin{split}
		I_2(R)&=\frac{4}{\pi}\int_{-Ra_R}^{+\infty}\int_{-Ra_R}^{+\infty}\int_{0}^{\pi\sqrt{(R+\frac{x}{a_R})(R+\frac{y}{a_R})}}f(x)f(y)e^{-\frac{(x-y)^2}{a_R^2}}\sqrt{\Big(1+\frac{x}{Ra_R}\Big)\Big(1+\frac{y}{Ra_R}\Big)}\\
		&\qquad\qquad\qquad\qquad\qquad\qquad\quad\cdot e^{-4(R+\frac{x}{a_R})(R+\frac{y}{a_R})\sin^2{\frac{t}{2\sqrt{(R+\frac{x}{a_R})(R+\frac{y}{a_R})}}}}\mathrm{d}t\mathrm{d}x\mathrm{d}y\cdot\frac{R}{a_R^2}.
	\end{split}
\end{equation*}

Consider the integral
\begin{equation*}
	\begin{split}
		&\int_{\mathbb{R}^3}f^2(x)e^{-(\frac{x}{a_R}-u)^2}\sqrt{\Big(1+\frac{x}{Ra_R}\Big)\Big(1+\frac{u}{R}\Big)}\,e^{-4(R+\frac{x}{a_R})(R+u)\sin^2{\frac{t}{2\sqrt{(R+\frac{x}{a_R})(R+u)}}}}\\
		&\qquad\qquad\qquad\quad\,\,\cdot\chi_{\left[0,\pi\sqrt{(R+\frac{x}{a_R})(R+u)}\right]}(t)\chi_{[-Ra_R,+\infty)}(x)\chi_{[-R,+\infty)}(u)\mathrm{d}t\mathrm{d}x\mathrm{d}u.
	\end{split}
\end{equation*}
When $R$ is sufficiently large, the integrand is dominated by
\begin{equation*}
	\begin{split}
		e^{2l^2}f^2(x)e^{-(|u|-l)^2}\sqrt{(1+|x|)(1+|u|)}\,e^{-\frac{4}{\pi^2}t^2}\chi_{[0,+\infty)}(t),
	\end{split}
\end{equation*}
where $l=\max\limits_{x\in\supp f}|x|$. Here we used the fact that when $x\in\supp f$,
\begin{equation*}
	\begin{split}
		e^{-(\frac{x}{a_R}-u)^2}=e^{-u^2+2u\frac{x}{a_R}-\frac{x^2}{a_R^2}}\leq e^{-u^2+2|u||x|+|x|^2}\leq e^{-u^2+2|u|l+l^2}=e^{2l^2}e^{-(|u|-l)^2}.
	\end{split}
\end{equation*}
Hence by the dominated convergence theorem, as $R\to+\infty$, the integral converges to
\begin{equation*}
	\begin{split}
		\int_{\mathbb{R}^3}f^2(x)e^{-u^2}e^{-t^2}\chi_{[0,+\infty)}(t)\mathrm{d}t\mathrm{d}x\mathrm{d}u=\frac{\pi}{2}\int_{\mathbb{R}}f^2(x)\mathrm{d}x.
	\end{split}
\end{equation*}
It follows that when $R\to+\infty$,
\begin{equation*}
	\begin{split}
		I_1(R)\sim\frac{2R}{a_R}\int_{\mathbb{R}}f^2(x)\mathrm{d}x.
	\end{split}
\end{equation*}

Consider the integral
\begin{equation*}
	\begin{split}
		&\int_{\mathbb{R}^3}f(x)f(y)e^{-\frac{(x-y)^2}{a_R^2}}\sqrt{\Big(1+\frac{x}{Ra_R}\Big)\Big(1+\frac{y}{Ra_R}\Big)}\,e^{-4(R+\frac{x}{a_R})(R+\frac{y}{a_R})\sin^2{\frac{t}{2\sqrt{(R+\frac{x}{a_R})(R+\frac{y}{a_R})}}}}\\
		&\qquad\qquad\qquad\qquad\quad\,\,\cdot\chi_{\left[0,\pi\sqrt{(R+\frac{x}{a_R})(R+\frac{y}{a_R})}\right]}(t)\chi_{[-Ra_R,+\infty)}(x)\chi_{[-Ra_R,+\infty)}(u)\mathrm{d}t\mathrm{d}x\mathrm{d}y.
	\end{split}
\end{equation*}
When $R$ is sufficiently large, the integrand is dominated by
\begin{equation*}
	\begin{split}
		|f(x)f(y)|\sqrt{(1+|x|)(1+|y|)}\,e^{-\frac{4}{\pi^2}t^2}\chi_{[0,+\infty)}(t).
	\end{split}
\end{equation*}
Hence by the dominated convergence theorem, as $R\to+\infty$, the integral converges to
\begin{equation*}
	\begin{split}
		\int_{\mathbb{R}^3}f(x)f(y)e^{-t^2}\chi_{[0,+\infty)}(t)\mathrm{d}t\mathrm{d}x\mathrm{d}y=\frac{\sqrt{\pi}}{2}\left(\int_{\mathbb{R}}f(x)\mathrm{d}x\right)^2.
	\end{split}
\end{equation*}
It follows that when $R\to+\infty$,
\begin{equation*}
	\begin{split}
		I_2(R)\sim\frac{2R}{\sqrt{\pi}a_R^2}\left(\int_{\mathbb{R}}f(x)\mathrm{d}x\right)^2=o\Big(\frac{R}{a_R}\Big).
	\end{split}
\end{equation*}

Therefore, when $R\to+\infty$,
\begin{equation*}
	\begin{split}
		\Var\Big(\sum_{z\in\mathcal{G}}f\big(a_R(|z|-R)\big)\Big)\sim\frac{2R}{a_R}\int_{\mathbb{R}}f^2(x)\mathrm{d}x.
	\end{split}
\end{equation*}

In the case $1\ll a_R\ll R$ as $R\to+\infty$, applying Soshnikov's Theorem~\ref{Soshnikov}, the above calculations yield that
\begin{equation*}
	\begin{split}
		\frac{\sum_{z\in\mathcal{G}}f\big(a_R(|z|-R)\big)-\mathbb{E}\big[\sum_{z\in\mathcal{G}}f\big(a_R(|z|-R)\big)\big]}{\sqrt{\Var\big(\sum_{z\in\mathcal{G}}f\big(a_R(|z|-R)\big)\big) }}\xrightarrow[R\to+\infty]{\mathrm{law}}\mathcal{N}(0,1).
	\end{split}
\end{equation*}
This gives that
\begin{equation*}
	\begin{split}
		\sqrt{\frac{a_R}{R}}\sum_{z\in\mathcal{G}}f\big(a_R(|z|-R)\big)-\sqrt{\frac{a_R}{R}}\mathbb{E}\Big[\sum_{z\in\mathcal{G}}f\big(a_R(|z|-R)\big)\Big]\xrightarrow[R\to+\infty]{\mathrm{law}}\mathcal{N}\Big(0,2\int_{\mathbb{R}}f^2(x)\mathrm{d}x\Big).
	\end{split}
\end{equation*}
Moreover,
\begin{equation*}
	\begin{split}
		\sqrt{\frac{a_R}{R}}\mathbb{E}\Big[\sum_{z\in\mathcal{G}}f\big(a_R(|z|-R)\big)\Big]-2\sqrt{\frac{R}{a_R}}\int_{\mathbb{R}}f(x)\mathrm{d}x\xrightarrow[R\to+\infty]{}0,
	\end{split}
\end{equation*}
hence
\begin{equation*}
	\begin{split}
		\sqrt{\frac{a_R}{R}}\sum_{z\in\mathcal{G}}f\big(a_R(|z|-R)\big)-2\sqrt{\frac{R}{a_R}}\int_{\mathbb{R}}f(x)\mathrm{d}x\xrightarrow[R\to+\infty]{\mathrm{law}}\mathcal{N}\Big(0,2\int_{\mathbb{R}}f^2(x)\mathrm{d}x\Big).
	\end{split}
\end{equation*}
This completes the proof of Theorem~\ref{White Noise Ginibre}.

\section{Proof of Theorem~\ref{Hyperbolic Poisson} and Theorem~\ref{Ginibre Poisson}}
\subsection{A Lemma of convergence towards the Poisson point process}
The main reason why a Poisson point process appears is because we are dealing with independent particles. Nevertheless, this is not enough. We need that each particle escapes every bounded set and we need this to be done in a uniform way. More precisely, we use the following lemma.

\begin{lemma}\label{lemma}
	For each $R>0$, let $\{X_n^{(R)}:n\in\mathbb{N}\}$ be a sequence of independent real-valued random variables. Suppose that for every compact $K\subset\mathbb{R}$,
	\begin{equation*}
		\begin{split}
			\sup_{n\geq0}\mathbb{P}(X_n^{(R)}\in K)\xrightarrow[R\to+\infty]{}0,
		\end{split}
	\end{equation*}
	and suppose that there is a positive Radon measure $\nu$ on $\mathbb{R}$ such that
	\begin{equation*}
		\begin{split}
			\sum_{n=0}^{\infty}\mathbb{E}\big[f(X_n^{(R)})\big]\xrightarrow[R\to+\infty]{}\int_{\mathbb{R}}f\mathrm{d}\nu
		\end{split}
	\end{equation*}
	for every measurable compactly supported function $f:\mathbb{R}\to[0,1)$. Then for every measurable compactly supported function $f:\mathbb{R}\to[0,1)$,
	\begin{equation*}
		\begin{split}
			\mathbb{E}\left[\prod_{n=0}^{\infty}\big(1-f(X_n^{(R)})\big)\right]\xrightarrow[R\to+\infty]{}\exp\left(-\int_{\mathbb{R}}f\mathrm{d}\nu\right).
		\end{split}
	\end{equation*}
\end{lemma}

\begin{proof}[Proof of Lemma~\ref{lemma}]
For every measurable compactly supported function $f:\mathbb{R}\to[0,1)$, by independence, we have
\begin{equation*}
	\begin{split}
		\mathbb{E}\left[\prod_{n=0}^{\infty}\big(1-f(X_n^{(R)})\big)\right]=\prod_{n=0}^{\infty}\mathbb{E}\big[1-f(X_n^{(R)})\big]=\prod_{n=0}^{\infty}\Big(1-\mathbb{E}\big[f(X_n^{(R)})\big]\Big),
	\end{split}
\end{equation*}
and then
\begin{equation*}
	\begin{split}
		\log\mathbb{E}\left[\prod_{n=0}^{\infty}\big(1-f(X_n^{(R)})\big)\right]=\sum_{n=0}^{\infty}\log\Big(1-\mathbb{E}\big[f(X_n^{(R)})\big]\Big).
	\end{split}
\end{equation*}

We can use the fact that $\log(1-x)=-x+\Theta(x)$, where $\Theta(x)=O(x^2)$ as $x\to0$, to obtain that
\begin{equation*}
	\begin{split}
		\sum_{n=0}^{\infty}\log\Big(1-\mathbb{E}\big[f(X_n^{(R)})\big]\Big)=-\sum_{n=0}^{\infty}\mathbb{E}\big[f(X_n^{(R)})\big]+\sum_{n=0}^{\infty}\Theta\Big(\mathbb{E}\big[f(X_n^{(R)})\big]\Big).
	\end{split}
\end{equation*}
Note that
\begin{equation*}
	\begin{split}
		\sup_{n\geq0}\mathbb{E}\big[f(X_n^{(R)})\big]\leq\sup_{n\geq0}\mathbb{P}(X_n^{(R)}\in\supp f)\xrightarrow[R\to+\infty]{}0,
	\end{split}
\end{equation*}
and the convergence of $\sum_{n=0}^{\infty}\mathbb{E}\big[f(X_n^{(R)})\big]$,
we get
\begin{equation*}
	\begin{split}
		\sum_{n=0}^{\infty}\Theta\Big(\mathbb{E}\big[f(X_n^{(R)})\big]\Big)=O\left(\sum_{n=0}^{\infty}\mathbb{E}\big[f(X_n^{(R)})\big]^2\right)=O\left(\sup_{n\geq0}\mathbb{E}\big[f(X_n^{(R)})\big]\right)\xrightarrow[R\to+\infty]{}0.
	\end{split}
\end{equation*}
This implies that
\begin{equation*}
	\begin{split}
		\sum_{n=0}^{\infty}\log\Big(1-\mathbb{E}\big[f(X_n^{(R)})\big]\Big)\xrightarrow[R\to+\infty]{}-\int_{\mathbb{R}}f\mathrm{d}\nu.
	\end{split}
\end{equation*}
Therefore,
\begin{equation*}
	\begin{split}
		\mathbb{E}\left[\prod_{n=0}^{\infty}\big(1-f(X_n^{(R)})\big)\right]\xrightarrow[R\to+\infty]{}\exp\left(-\int_{\mathbb{R}}f\mathrm{d}\nu\right).
	\end{split}
\end{equation*}

This completes the proof of Lemma~\ref{lemma}.
\end{proof}

\begin{corollary}[Convergence towards a PPP]\label{corollary}
	Under the conditions of Lemma~\ref{lemma}. If, moreover, $\mathcal{X}^{(R)}=\{X_n^{(R)}:n\in\mathbb{N}\}$ is a point process on $\mathbb{R}$ for each $R>0$. Then
	\begin{equation*}
		\begin{split}
			\mathcal{X}^{(R)}\xrightarrow[R\to+\infty]{\mathrm{law}}\mathcal{P}_{\nu},
		\end{split}
	\end{equation*}
	where $\mathcal{P}_{\nu}$ is the Poisson point process on $\mathbb{R}$ with mean measure or intensity $\nu$.
\end{corollary}

\subsection{Proof of Theorem~\ref{Hyperbolic Poisson} (hyperbolic case)}
Recall that
$(\rho_n^{(\alpha)})_{n\geq0}$, $\alpha>0$, is a family of non-negative independent random variables such that
\begin{equation*}
	\begin{split}
		\rho_n^{(\alpha)}\sim2\frac{\Gamma(\alpha+n+1)}{\Gamma(\alpha)\Gamma(n+1)}r^{2n+1}(1-r^2)^{\alpha-1}\mathrm{d}r.
	\end{split}
\end{equation*}
Fix $\alpha>0$, if we define $X_n^{(R)}=e^R(|\rho_n^{(\alpha)}|_{\mathrm{h}}-R)$ for each $R>0$ and $n\in\mathbb{N}$, we can obtain that
\begin{equation*}
	\begin{split}
		\{e^{R}(|z|_{\mathrm{h}}-R):z\in\mathcal{H}_{\alpha}\}\sim\{X_n^{(R)}:n\in\mathbb{N}\}.
	\end{split}
\end{equation*}
We are going to show that
\begin{equation*}
	\begin{split}
		\{X_n^{(R)}:n\in\mathbb{N}\}\xrightarrow[R\to+\infty]{\mathrm{law}}\mathcal{P}_{\alpha/4}.
	\end{split}
\end{equation*}

For every measurable compactly supported function $f:\mathbb{R}\to[0,1)$, a similar argument as Section~\ref{Proof of White Noise Hyperbolic} gives that
\begin{equation*}
	\begin{split}
		&\quad\,\sum_{n=0}^{\infty}\mathbb{E}\big[f(X_n^{(R)})\big]\\
		&=4^{\alpha}\int_{-Re^R}^{+\infty}\int_{0}^{+\infty}f(x)e^{x/e^R}k_{\lfloor te^{R+x/e^R}\rfloor}^{(\alpha)}\Big(1-\frac{2}{e^{R+x/e^R}+1}\Big)^{2{\lfloor te^{R+x/e^R}\rfloor}+1}\frac{e^{\alpha(R+x/e^R)}}{(e^{R+x/e^R}+1)^{2\alpha}}\mathrm{d}t\mathrm{d}x\\
		&\xrightarrow[R\to+\infty]{}\frac{\alpha}{4}\int_{\mathbb{R}}f(x)\mathrm{d}x.
	\end{split}
\end{equation*}

Hence by Corollary~\ref{corollary}, we shall prove that for every $T>0$,
\begin{equation*}
	\begin{split}
		\sup_{n\geq0}\mathbb{P}(X_n^{(R)}\in[-T,T])\xrightarrow[R\to+\infty]{}0.
	\end{split}
\end{equation*}
When $R$ is large enough, we have
\begin{equation*}
	\begin{split}
		\mathbb{P}(X_n^{(R)}\in [-T,T])&=4^{\alpha}\int_{-T}^{T}\frac{\Gamma(\alpha+n+1)}{\Gamma(\alpha)\Gamma(n+1)}\Big(1-\frac{2}{e^{R+x/e^R}+1}\Big)^{2n+1}\frac{e^{\alpha(R+x/e^R)}}{(e^{R+x/e^R}+1)^{2\alpha}}\mathrm{d}x\cdot\frac{1}{e^R}\\
		&\leq\frac{2^{2\alpha+1}T}{e^R}\cdot\frac{\Gamma(\alpha+n+1)}{\Gamma(\alpha)\Gamma(n+1)}\Big(1-\frac{2}{e^{R+T/e^R}+1}\Big)^{2n+1}\frac{1}{e^{\alpha(R-T/e^R)}}\\
		&\leq\frac{2^{4\alpha+1}T}{e^R}\cdot\frac{\Gamma(\alpha+n+1)}{\Gamma(\alpha)\Gamma(n+1)}\Big(1-\frac{1}{2e^R}\Big)^{2n+1}\frac{1}{e^{\alpha R}}.
	\end{split}
\end{equation*}
Notice that there exists a constant $C>0$ depending only on $\alpha$ such that for any $n\in\mathbb{N}$,
\begin{equation*}
	\begin{split}
		\frac{\Gamma(\alpha+n+1)}{\Gamma(\alpha)\Gamma(n+1)}\leq Cn^{\alpha},
	\end{split}
\end{equation*}
and
\begin{equation*}
	\begin{split}
		\Big(1-\frac{1}{2e^R}\Big)^{2n+1}=\Big[\Big(1+\frac{1}{2e^R-1}\Big)^{2e^R}\Big]^{-\frac{2n+1}{2e^R}}\leq e^{-n/e^R}.
	\end{split}
\end{equation*}
Hence
\begin{equation*}
	\begin{split}
		\mathbb{P}(X_n^{(R)}\in[-T,T])\leq\frac{2^{4\alpha+1}CT}{e^R}(n/e^R)^{\alpha}e^{-n/e^R}.
	\end{split}
\end{equation*}
Since $x^{\alpha}e^{-x}$ is bounded for $x>0$, we get
\begin{equation*}
	\begin{split}
		\sup_{n\geq0}\mathbb{P}(X_n^{(R)}\in[-T,T])=O(e^{-R}).
	\end{split}
\end{equation*}

This completes the proof of Theorem~\ref{Hyperbolic Poisson}.

\subsection{Proof of Theorem~\ref{Ginibre Poisson} (Ginibre case)}
Recall that
$(\rho_n)_{n\geq0}$ is a family of non-negative independent random variables such that
\begin{equation*}
	\begin{split}
		\rho_n\sim\frac{2r^{2n+1}e^{-r^2}}{n!}\mathrm{d}r.
	\end{split}
\end{equation*}
If we define $X_n^{(R)}=R(\rho_n-R)$ for each $R>0$ and $n\in\mathbb{N}$, we can obtain that
\begin{equation*}
	\begin{split}
		\{R(|z|-R):z\in\mathcal{G}\}\sim\{X_n^{(R)}:n\in\mathbb{N}\}.
	\end{split}
\end{equation*}
We are going to show that
\begin{equation*}
	\begin{split}
		\{X_n^{(R)}:n\in\mathbb{N}\}\xrightarrow[R\to+\infty]{\mathrm{law}}\mathcal{P}_2.
	\end{split}
\end{equation*}

For every measurable compactly supported function $f:\mathbb{R}\to[0,1)$, a similar argument as Section~\ref{Proof of White Noise Ginibre} gives that
\begin{equation*}
	\begin{split}
		\sum_{n=0}^{\infty}\mathbb{E}\big[f(X_n^{(R)})\big]=2\int_{-R^2}^{+\infty}f(x)\mathrm{d}x+2\int_{-R^2}^{+\infty}xf(x)\mathrm{d}x\cdot\frac{1}{R^2}\xrightarrow[R\to+\infty]{}2\int_{\mathbb{R}}f(x)\mathrm{d}x.
	\end{split}
\end{equation*}

Hence by Corollary~\ref{corollary}, we shall prove that for every $T>0$,
\begin{equation*}
	\begin{split}
		\sup_{n\geq0}\mathbb{P}(X_n^{(R)}\in[-T,T])\xrightarrow[R\to+\infty]{}0.
	\end{split}
\end{equation*}
When $R$ is large enough, we have
\begin{equation*}
	\begin{split}
		\mathbb{P}(X_n^{(R)}\in [-T,T])&=2\int_{-T}^{T}\frac{(R+\frac{x}{R})^{2n+1}}{n!}e^{-(R+\frac{x}{R})}\mathrm{d}x\cdot\frac{1}{R}\\
		&\leq4\frac{R^{2n}(1+\frac{T}{R^2})^{2n}}{n!}e^{-R^2}\int_{-T}^{T}e^{-2x-\frac{x^2}{R^2}}\mathrm{d}x\\
		&\leq8Te^{2T}\frac{R^{2n}e^{\frac{2nT}{R^2}}}{n!}e^{-R^2}\\
		&\leq8Te^{2T}\frac{\big(\lceil R^2e^{\frac{2T}{R^2}}\rceil\big)^n}{n!}e^{-R^2}.
	\end{split}
\end{equation*}
Notice that
\begin{equation*}
	\begin{split}
		\frac{\lceil R^2e^{\frac{2T}{R^2}}\rceil^n}{n!}\leq\frac{\big(\lceil R^2e^{\frac{2T}{R^2}}\rceil\big)^{\lceil R^2e^{\frac{2T}{R^2}}\rceil}}{\lceil R^2e^{\frac{2T}{R^2}}\rceil!}=O\left(\frac{e^{\lceil R^2e^{\frac{2T}{R^2}}\rceil}}{\sqrt{\lceil R^2e^{\frac{2T}{R^2}}\rceil}}\right)=O\Big(\frac{e^{R^2}}{R}\Big),
	\end{split}
\end{equation*}
where we used the Stirling's approximation. It follows that
\begin{equation*}
	\begin{split}
		\sup_{n\geq0}\mathbb{P}(X_n^{(R)}\in[-T,T])=O(R^{-1}).
	\end{split}
\end{equation*}

This completes the proof of Theorem~\ref{Ginibre Poisson}.

\section{The case of superexponential growth}
\subsection{Fluctuations in the hyperbolic case}
Theorem~\ref{Hyperbolic Poisson} explains that the limiting behaviour of the point process $\{e^{R}(|z|_{\mathrm{h}}-R):z\in\mathcal{H}_{\alpha}\}$ is Poisson when $R\to+\infty$. For a bounded measurable and compactly supported function $f: \mathbb{R}\to\mathbb{R}$, Theorem~\ref{Extreme Hyperbolic Normality} illustrates that the limiting behaviour of $\sum_{z\in\mathcal{H}_{\alpha}}f(|z|_{\mathrm{h}}-R)$ is Gaussian when $R\to+\infty$. In the case $1\ll a_R\ll e^R$ as $R\to+\infty$, Theorem~\ref{White Noise Hyperbolic} shows that the limiting behaviour of $\sum_{z\in\mathcal{H}_{\alpha}}f\big(a_R(|z|_{\mathrm{h}}-R)\big)$ is also Gaussian. For completeness, we continue to consider the limiting behaviour of $\sum_{z\in\mathcal{H}_{\alpha}}f\big(a_R(|z|_{\mathrm{h}}-R)\big)$ in the case $a_R\gg e^R$ and $a_R\ll 1$ as $R\to+\infty$.

In the case $a_R\gg e^R$ as $R\to+\infty$, for a bounded measurable and compactly supported function $f: \mathbb{R}\to\mathbb{R}$, the limiting behaviour of $\sum_{z\in\mathcal{H}_{\alpha}}f\big(a_R(|z|_{\mathrm{h}}-R)\big)$ is zero. This can be seen by
\begin{equation*}
	\begin{split}
		\Var\Big(\sum_{z\in\mathcal{H}_{\alpha}}f\big(a_R(|z|_{\mathrm{h}}-R)\big)\Big)\xrightarrow[R\to+\infty]{}0.
	\end{split}
\end{equation*}
In fact, a similar argument as in Section~\ref{Proof of White Noise Hyperbolic} implies that
\begin{equation*}
	\begin{split}
		\Var\Big(\sum_{z\in\mathcal{H}_{\alpha}}f\big(a_R(|z|_{\mathrm{h}}-R)\big)\Big)\leq\sum_{z\in\mathcal{H}_{\alpha}}\mathbb{E}\Big[f^2\big(a_R(|z|_{\mathrm{h}}-R)\big)\Big]\sim\frac{\alpha e^R}{4a_R}\int_{\mathbb{R}}f^2(x)\mathrm{d}x\xrightarrow[R\to+\infty]{}0.
	\end{split}
\end{equation*}

We now turn to consider the case $a_R\ll 1$ as $R\to+\infty$. For a bounded measurable and compactly supported function $f: \mathbb{R}\to\mathbb{R}$, denote $M_f$ the essential right endpoint of $\supp f$, that is, 
\begin{equation*}
	\begin{split}
		M_f:=\inf\big\{M\in\mathbb{R}\,\big|\,f(x)=0\text{ almost everywhere on }[M,+\infty)\big\}.
	\end{split}
\end{equation*}

\begin{theorem}\label{Small White Noise Hyperbolic}
	Suppose that $a_R>0$ satisfies $a_R\ll 1$ as $R\to+\infty$. Let $f$ be a real-valued bounded measurable function on $\mathbb{R}$ with compact support such that $f(M_f^-):=\lim\limits_{x\to M_f^-}f(x)$ exists and is non-zero. If $R+M_f/a_R\to+\infty$ as $R\to+\infty$, then for each $\alpha>0$,
	\begin{equation*}
		\begin{split}
			\frac{\sum_{z\in\mathcal{H}_{\alpha}}f\big(a_R(|z|_{\mathrm{h}}-R)\big)-\mathbb{E}\big[\sum_{z\in\mathcal{H}_{\alpha}}f\big(a_R(|z|_{\mathrm{h}}-R)\big)\big]}{\sqrt{\Var\big(\sum_{z\in\mathcal{H}_{\alpha}}f\big(a_R(|z|_{\mathrm{h}}-R)\big)\big)}}\xrightarrow[R\to+\infty]{\mathrm{law}}\mathcal{N}(0,1).
		\end{split}
	\end{equation*}
\end{theorem}

With the help of Theorem~\ref{Small White Noise Hyperbolic}, we can do more detailed discussion when $a_R\ll 1$ as $R\to+\infty$.

(i) In the case $R^{-1}\ll a_R\ll 1$, we always have $R+M_f/a_R\to+\infty$, so the limiting behaviour of $\sum_{z\in\mathcal{H}_{\alpha}}f\big(a_R(|z|_{\mathrm{h}}-R)\big)$ is Gaussian when $f(M_f^-)$ exists and is non-zero.

(ii) In the case $a_R=R^{-1}$, if $M_f\leq-1$, since $a_R(|z|_{\mathrm{h}}-R)>-1$ except $z=0$, $\sum_{z\in\mathcal{H}_{\alpha}}f\big(R^{-1}(|z|_{\mathrm{h}}-R)\big)$ is almost surely the zero random variable for every $R>0$; if $M_f>-1$, we have $R+M_f/a_R\to+\infty$, so the limiting behaviour of $\sum_{z\in\mathcal{H}_{\alpha}}f\big(R^{-1}(|z|_{\mathrm{h}}-R)\big)$ is Gaussian when $f(M_f^-)$ exists and is non-zero.

(iii) In the case $a_R\ll R^{-1}$, if $M_f<0$, since $a_R(|z|_{\mathrm{h}}-R)>-Ra_R\to0$  except $z=0$, $\sum_{z\in\mathcal{H}_{\alpha}}f\big(a_R(|z|_{\mathrm{h}}-R)\big)$ is almost surely the zero random variable for sufficiently large $R$; if $M_f\geq0$, we have $R+M_f/a_R\to+\infty$, so the limiting behaviour of $\sum_{z\in\mathcal{H}_{\alpha}}f\big(a_R(|z|_{\mathrm{h}}-R)\big)$ is Gaussian when $f(M_f^-)$ exists and is non-zero.

\begin{question}
	For the hyperbolic situation, in the case $a_R=R^{-1}$ and $M_f>-1$, or $a_R\ll R^{-1}$ and $M_f\geq0$, does the central limit theorem also holds without the condition that $f(M_f^-)$ exists and is non-zero?
\end{question}

\begin{proof}[Proof of Theorem~\ref{Small White Noise Hyperbolic}]
We will use Soshnikov's Theorem~\ref{Soshnikov} to prove this theorem. The calculations are similar as Section~\ref{Proof of White Noise Hyperbolic}.

Let us first calculate the expectation
\begin{equation*}
	\begin{split}
		&\quad\,\mathbb{E}\Big[\sum_{z\in\mathcal{H}_{\alpha}}\big|f\big(a_R(|z|_{\mathrm{h}}-R)\big)\big|\Big]\\
		&=4^{\alpha}\int_{-Ra_R}^{M_f}\int_{0}^{+\infty}|f(x)|e^{\frac{x}{a_R}}k_{\lfloor te^{R+\frac{x}{a_R}}\rfloor}^{(\alpha)}\Big(1-\frac{2}{e^{R+\frac{x}{a_R}}+1}\Big)^{2{\lfloor te^{R+\frac{x}{a_R}}\rfloor}+1}\frac{e^{\alpha(R+\frac{x}{a_R})}}{(e^{R+\frac{x}{a_R}}+1)^{2\alpha}}\mathrm{d}t\mathrm{d}x\cdot\frac{e^R}{a_R}.
	\end{split}
\end{equation*}
Make a variable substitution by $x=a_Ry+M_f$, then
\begin{equation*}
	\begin{split}
		&\quad\,\mathbb{E}\Big[\sum_{z\in\mathcal{H}_{\alpha}}\big|f\big(a_R(|z|_{\mathrm{h}}-R)\big)\big|\Big]\\
		&=4^{\alpha}\int_{-R-M_f/a_R}^{0}\int_{0}^{+\infty}|f(a_Ry+M_f)|e^yk_{\lfloor te^{R+M_f/a_R+y}\rfloor}^{(\alpha)}\Big(1-\frac{2}{e^{R+M_f/a_R+y}+1}\Big)^{2{\lfloor te^{R+M_f/a_R+y}\rfloor}+1}\\
		&\qquad\qquad\qquad\qquad\qquad\qquad\qquad\qquad\qquad\qquad\qquad\cdot\frac{e^{\alpha(R+M_f/a_R+y)}}{(e^{R+M_f/a_R+y}+1)^{2\alpha}}\mathrm{d}t\mathrm{d}y\cdot e^{R+M_f/a_R}.
	\end{split}
\end{equation*}
Hence
\begin{equation*}
	\begin{split}
		\mathbb{E}\Big[\sum_{z\in\mathcal{H}_{\alpha}}\big|f\big(a_R(|z|_{\mathrm{h}}-R)\big)\big|\Big]\sim\frac{\alpha|f(M_f^-)|}{4}e^{R+M_f/a_R}.
	\end{split}
\end{equation*}

Next we shall calculate the variance
\begin{equation*}
	\begin{split}
		\Var\Big(\sum_{z\in\mathcal{H}_{\alpha}}f\big(a_R(|z|_{\mathrm{h}}-R)\big)\Big)=\sum_{n=0}^{\infty}\mathbb{E}\Big[f^2\big(a_R(|\rho_n^{(\alpha)}|_{\mathrm{h}}-R)\big)\Big]-\sum_{n=0}^{\infty}\mathbb{E}\Big[f\big(a_R(|\rho_n^{(\alpha)}|_{\mathrm{h}}-R)\big)\Big]^2.
	\end{split}
\end{equation*}
For the first term, we have
\begin{equation*}
	\begin{split}
		\sum_{n=0}^{\infty}\mathbb{E}\Big[f^2\big(a_R(|\rho_n^{(\alpha)}|_{\mathrm{h}}-R)\big)\Big]\sim\frac{\alpha|f(M_f^-)|^2}{4}e^{R+M_f/a_R}.
	\end{split}
\end{equation*}
As for the second term
\begin{equation*}
	\begin{split}
		&\quad\,\sum_{n=0}^{\infty}\mathbb{E}\Big[f\big(a_R(|\rho_n^{(\alpha)}|_{\mathrm{h}}-R)\big)\Big]^2\\
		&=4^{2\alpha}\sum_{n=0}^{\infty}\left(\int_{-Ra_R}^{M_f}f(x)k_n^{(\alpha)}\Big(1-\frac{2}{e^{R+\frac{x}{a_R}}+1}\Big)^{2n+1}\frac{e^{\alpha(R+\frac{x}{a_R})}}{(e^{R+\frac{x}{a_R}}+1)^{2\alpha}}\mathrm{d}x\right)^2\cdot\frac{1}{a_R^2},
	\end{split}
\end{equation*}
set $n=\lfloor te^{R+M_f/a_R}\rfloor$, then
\begin{equation*}
	\begin{split}
		&\quad\,\sum_{n=0}^{\infty}\mathbb{E}\Big[f\big(a_R(|\rho_n^{(\alpha)}|_{\mathrm{h}}-R)\big)\Big]^2\\
		&=4^{2\alpha}\int_{0}^{+\infty}\left(\int_{-Ra_R}^{M_f}f(x)k_{\lfloor te^{R+M_f/a_R}\rfloor}^{(\alpha)}\Big(1-\frac{2}{e^{R+\frac{x}{a_R}}+1}\Big)^{2\lfloor te^{R+M_f/a_R}\rfloor+1}\right.\\
		&\left.\qquad\qquad\qquad\qquad\qquad\qquad\qquad\qquad\cdot\frac{e^{\alpha(R+\frac{x}{a_R})}}{(e^{R+\frac{x}{a_R}}+1)^{2\alpha}}\mathrm{d}x\right)^2\mathrm{d}t\cdot\frac{e^{R+M_f/a_R}}{a_R^2}.
	\end{split}
\end{equation*}
Make a variable substitution by $x=a_Ry+M_f$, so
\begin{equation*}
	\begin{split}
		&\quad\,\sum_{n=0}^{\infty}\mathbb{E}\Big[f\big(a_R(|\rho_n^{(\alpha)}|_{\mathrm{h}}-R)\big)\Big]^2\\
		&=4^{2\alpha}\int_{0}^{+\infty}\left(\int_{-R-M_f/a_R}^{0}f(a_Ry+M_f)k_{\lfloor te^{R+M_f/a_R}\rfloor}^{(\alpha)}\Big(1-\frac{2}{e^{R+M_f/a_R+y}+1}\Big)^{2\lfloor te^{R+M_f/a_R}\rfloor+1}\right.\\
		&\left.\qquad\qquad\qquad\qquad\qquad\qquad\qquad\qquad\qquad\qquad\cdot\frac{e^{\alpha(R+M_f/a_R+y)}}{(e^{R+M_f/a_R+y}+1)^{2\alpha}}\mathrm{d}y\right)^2\mathrm{d}t\cdot{e^{R+M_f/a_R}}\\
		&\sim\frac{\alpha|f(M_f^-)|^2}{4B(\alpha,\alpha+1)}\int_{-\infty}^0\int_{-\infty}^0\frac{e^{(\alpha+1)(x+y)}}{(e^x+e^y)^{2\alpha+1}}\mathrm{d}x\mathrm{d}y\cdot e^{R+M_f/a_R}.
	\end{split}
\end{equation*}
Notice taht
\begin{equation*}
	\begin{split}
		\frac{1}{B(\alpha,\alpha+1)}\int_{-\infty}^{+\infty}\int_{-\infty}^0\frac{e^{(\alpha+1)(x+y)}}{(e^x+e^y)^{2\alpha+1}}\mathrm{d}x\mathrm{d}y=1,
	\end{split}
\end{equation*}
we conclude that
\begin{equation*}
	\begin{split}
		\Var\Big(\sum_{z\in\mathcal{H}_{\alpha}}f\big(a_R(|z|_{\mathrm{h}}-R)\big)\Big)\sim\frac{\alpha|f(M_f^-)|^2}{4B(\alpha,\alpha+1)}\int_0^{+\infty}\int_{-\infty}^0\frac{e^{(\alpha+1)(x+y)}}{(e^x+e^y)^{2\alpha+1}}\mathrm{d}x\mathrm{d}y\cdot e^{R+M_f/a_R}.
	\end{split}
\end{equation*}

The above calculations yield that when $R\to+\infty$,
\begin{equation*}
	\begin{split}
		\mathbb{E}\Big[\sum_{z\in\mathcal{H}_{\alpha}}\big|f\big(a_R(|z|_{\mathrm{h}}-R)\big)\big|\Big]=O\left(\Var\Big(\sum_{z\in\mathcal{H}_{\alpha}}f\big(a_R(|z|_{\mathrm{h}}-R)\big)\Big)\right),
	\end{split}
\end{equation*}
and then we can use Soshnikov's Theorem~\ref{Soshnikov} to get the central limit theorem.

This completes the proof of Theorem~\ref{Small White Noise Hyperbolic}.
\end{proof}

\subsection{Fluctuations in the Ginibre case}
Theorem~\ref{Ginibre Poisson} explains that the limiting behaviour of the point process $\{R(|z|-R):z\in\mathcal{G}\}$ is Poisson when $R\to+\infty$. For a bounded measurable and compactly supported function $f: \mathbb{R}\to\mathbb{R}$, Theorem~\ref{Extreme Ginibre Normality} illustrates that the limiting behaviour of $\sum_{z\in\mathcal{G}}f(|z|-R)$ is Gaussian when $R\to+\infty$. In the case $1\ll a_R\ll R$ as $R\to+\infty$, Theorem~\ref{White Noise Ginibre} shows that the limiting behaviour of $\sum_{z\in\mathcal{G}}f\big(a_R(|z|-R)\big)$ is also Gaussian. For completeness, we continue to consider the limiting behaviour of $\sum_{z\in\mathcal{G}}f\big(a_R(|z|-R)\big)$ in the case $a_R\gg R$ and $a_R\ll 1$ as $R\to+\infty$.

In the case $a_R\gg R$ as $R\to+\infty$, for a bounded measurable and compactly supported function $f: \mathbb{R}\to\mathbb{R}$, the limiting behaviour of $\sum_{z\in\mathcal{G}}f\big(a_R(|z|-R)\big)$ is zero. This can be seen by
\begin{equation*}
	\begin{split}
		\Var\Big(\sum_{z\in\mathcal{G}}f\big(a_R(|z|-R)\big)\Big)\xrightarrow[R\to+\infty]{}0.
	\end{split}
\end{equation*}
In fact, a similar argument as Section~\ref{Proof of White Noise Ginibre} implies that
\begin{equation*}
	\begin{split}
		\Var\Big(\sum_{z\in\mathcal{G}}f\big(a_R(|z|-R)\big)\Big)\leq\sum_{z\in\mathcal{G}}\mathbb{E}\Big[f^2\big(a_R(|z|-R)\big)\Big]\sim\frac{2R}{a_R}\int_{\mathbb{R}}f^2(x)\mathrm{d}x\xrightarrow[R\to+\infty]{}0.
	\end{split}
\end{equation*}

We now turn to consider the case $a_R\ll 1$ as $R\to+\infty$.

\begin{theorem}\label{Small White Noise Ginibre}
	Suppose that $a_R>0$ satisfies $a_R\ll 1$ as $R\to+\infty$. Let $f$ be a real-valued bounded measurable function on $\mathbb{R}$ with compact support such that $f(M_f^-):=\lim\limits_{x\to M_f^-}f(x)$ exists and is non-zero. If $R+M_f/a_R\to+\infty$ as $R\to+\infty$, then,
	\begin{equation*}
		\begin{split}
			\frac{\sum_{z\in\mathcal{G}}f\big(a_R(|z|-R)\big)-\mathbb{E}\big[\sum_{z\in\mathcal{G}}f\big(a_R(|z|-R)\big)\big]}{\sqrt{\Var\big(\sum_{z\in\mathcal{G}}f\big(a_R(|z|-R)\big)\big)}}\xrightarrow[R\to+\infty]{\mathrm{law}}\mathcal{N}(0,1).
		\end{split}
	\end{equation*}
\end{theorem}

With the help of Theorem~\ref{Small White Noise Ginibre}, we can do more detailed discussion when $a_R\ll 1$ as $R\to+\infty$.

(i) In the case $R^{-1}\ll a_R\ll 1$, we always have $R+M_f/a_R\to+\infty$, so the limiting behaviour of $\sum_{z\in\mathcal{G}}f\big(a_R(|z|-R)\big)$ is Gaussian when $f(M_f^-)$ exists and is non-zero.

(ii) In the case $a_R=R^{-1}$, if $M_f\leq-1$, since $a_R(|z|-R)>-1$ except $z=0$, $\sum_{z\in\mathcal{G}}f\big(R^{-1}(|z|-R)\big)$ is almost surely the zero random variable for every $R>0$; if $M_f>-1$, we have $R+M_f/a_R\to+\infty$, so the limiting behaviour of $\sum_{z\in\mathcal{G}}f\big(R^{-1}(|z|-R)\big)$ is Gaussian when $f(M_f^-)$ exists and is non-zero.

(iii) In the case $a_R\ll R^{-1}$, if $M_f<0$, since $a_R(|z|-R)>-Ra_R\to0$  except $z=0$, $\sum_{z\in\mathcal{G}}f\big(a_R(|z|-R)\big)$ is almost surely the zero random variable for sufficiently large $R$; if $M_f\geq0$, we have $R+M_f/a_R\to+\infty$, so the limiting behaviour of $\sum_{z\in\mathcal{G}}f\big(a_R(|z|-R)\big)$ is Gaussian when $f(M_f^-)$ exists and is non-zero.

\begin{question}
	For the Ginibre situation, in the case $a_R=R^{-1}$ and $M_f>-1$, or $a_R\ll R^{-1}$ and $M_f\geq0$, does the central limit theorem also holds without the condition that $f(M_f^-)$ exists and is non-zero?
\end{question}

\begin{proof}[Proof of Theorem~\ref{Small White Noise Ginibre}]
We will use Soshnikov's Theorem~\ref{Soshnikov} to prove this theorem. The calculations are similar as Section~\ref{Proof of White Noise Ginibre}.
	
Let us first calculate the expectation
\begin{equation*}
	\begin{split}
		\mathbb{E}\Big[\sum_{z\in\mathcal{G}}\big|f\big(a_R(|z|-R)\big|\big)\Big]&=2\int_{-Ra_R}^{M_f}|f(x)|\mathrm{d}x\cdot\frac{R}{a_R}+2\int_{-Ra_R}^{M_f}x|f(x)|\mathrm{d}x\cdot\frac{1}{a_R^2}\\
		&=2\int_{-Ra_R}^{M_f}|f(x)|\big(R+\frac{x}{a_R}\big)\mathrm{d}x\cdot\frac{1}{a_R}.
	\end{split}
\end{equation*}
Make a variable substitution by $x=a_Ry+M_f$, then
\begin{equation*}
	\begin{split}
		\mathbb{E}\Big[\sum_{z\in\mathcal{G}}\big|f\big(a_R(|z|-R)\big|\big)\Big]&=2\int_{-R-M_f/a_R}^{0}|f(a_Ry+M_f)|\big(R+M_f/a_R+y\big)\mathrm{d}y\\
		&\leq2\left\|f\right\|_{\infty}(R+M_f/a_R)^2.
	\end{split}
\end{equation*}
	
Next we shall calculate the variance
\begin{equation*}
	\begin{split}
		&\quad\,\Var\Big(\sum_{z\in\mathcal{G}}f\big(a_R(|z|-R)\big)\Big)\\
		&=\frac{2}{\pi}\int_{-Ra_R}^{+\infty}\int_{-Ra_R}^{+\infty}\int_{0}^{\pi\sqrt{(R+\frac{x}{a_R})(R+\frac{y}{a_R})}}\big[f(x)-f(y)\big]^2e^{-\frac{(x-y)^2}{a_R^2}}\sqrt{\Big(R+\frac{x}{a_R}\Big)\Big(R+\frac{y}{a_R}\Big)}\\
		&\qquad\qquad\qquad\qquad\qquad\qquad\qquad\quad\cdot e^{-4(R+\frac{x}{a_R})(R+\frac{y}{a_R})\sin^2{\frac{t}{2\sqrt{(R+\frac{x}{a_R})(R+\frac{y}{a_R})}}}}\mathrm{d}t\mathrm{d}x\mathrm{d}y\cdot\frac{1}{a_R^2}\\
		&=I_1(R)+I_2(R),
	\end{split}
\end{equation*}
where
\begin{equation*}
	\begin{split}
		I_1(R)&=\frac{4}{\pi}\int_{-Ra_R}^{+\infty}\int_{-Ra_R}^{M_f}\int_{0}^{\pi\sqrt{(R+\frac{x}{a_R})(R+\frac{y}{a_R})}}f^2(x)e^{-\frac{(x-y)^2}{a_R^2}}\sqrt{\Big(R+\frac{x}{a_R}\Big)\Big(R+\frac{y}{a_R}\Big)}\\
		&\qquad\qquad\qquad\qquad\qquad\qquad\cdot e^{-4(R+\frac{x}{a_R})(R+\frac{y}{a_R})\sin^2{\frac{t}{2\sqrt{(R+\frac{x}{a_R})(R+\frac{y}{a_R})}}}}\mathrm{d}t\mathrm{d}x\mathrm{d}y\cdot\frac{1}{a_R^2},
	\end{split}
\end{equation*}
and
\begin{equation*}
	\begin{split}
		I_2(R)&=\frac{4}{\pi}\int_{-Ra_R}^{M_f}\int_{-Ra_R}^{M_f}\int_{0}^{\pi\sqrt{(R+\frac{x}{a_R})(R+\frac{y}{a_R})}}f(x)f(y)e^{-\frac{(x-y)^2}{a_R^2}}\sqrt{\Big(R+\frac{x}{a_R}\Big)\Big(R+\frac{y}{a_R}\Big)}\\
		&\qquad\qquad\qquad\qquad\qquad\qquad\cdot e^{-4(R+\frac{x}{a_R})(R+\frac{y}{a_R})\sin^2{\frac{t}{2\sqrt{(R+\frac{x}{a_R})(R+\frac{y}{a_R})}}}}\mathrm{d}t\mathrm{d}x\mathrm{d}y\cdot\frac{1}{a_R^2}\\
		&\leq\frac{4}{\pi}\int_{-Ra_R}^{M_f}\int_{-Ra_R}^{M_f}\int_{0}^{\pi\sqrt{(R+\frac{x}{a_R})(R+\frac{y}{a_R})}}f^2(x)e^{-\frac{(x-y)^2}{a_R^2}}\sqrt{\Big(R+\frac{x}{a_R}\Big)\Big(R+\frac{y}{a_R}\Big)}\\
		&\qquad\qquad\qquad\qquad\qquad\qquad\cdot e^{-4(R+\frac{x}{a_R})(R+\frac{y}{a_R})\sin^2{\frac{t}{2\sqrt{(R+\frac{x}{a_R})(R+\frac{y}{a_R})}}}}\mathrm{d}t\mathrm{d}x\mathrm{d}y\cdot\frac{1}{a_R^2}.
	\end{split}
\end{equation*}
It follows that
\begin{equation*}
	\begin{split}
		&\quad\,\Var\Big(\sum_{z\in\mathcal{G}}f\big(a_R(|z|-R)\big)\Big)\\
		&\geq\frac{4}{\pi}\int_{M_f}^{+\infty}\int_{-Ra_R}^{M_f}\int_{0}^{\pi\sqrt{(R+\frac{x}{a_R})(R+\frac{y}{a_R})}}f^2(x)e^{-\frac{(x-y)^2}{a_R^2}}\sqrt{\Big(R+\frac{x}{a_R}\Big)\Big(R+\frac{y}{a_R}\Big)}\\
		&\qquad\qquad\qquad\qquad\qquad\qquad\cdot e^{-4(R+\frac{x}{a_R})(R+\frac{y}{a_R})\sin^2{\frac{t}{2\sqrt{(R+\frac{x}{a_R})(R+\frac{y}{a_R})}}}}\mathrm{d}t\mathrm{d}x\mathrm{d}y\cdot\frac{1}{a_R^2}.
	\end{split}
\end{equation*}
Set $x=a_Ru+M_f$ and $y=a_Rv+M_f$, then
\begin{equation*}
	\begin{split}
		&\quad\,\Var\Big(\sum_{z\in\mathcal{G}}f\big(a_R(|z|-R)\big)\Big)\\
		&\geq\frac{4}{\pi}\int_{0}^{+\infty}\int_{-R-M_f/a_R}^{0}\int_{0}^{\pi\sqrt{(R+M_f/a_R+u)(R+M_f/a_R+v)}}f^2(a_Ru+M_f)e^{-(u-v)^2}\\
		&\qquad\qquad\qquad\qquad\qquad\qquad\cdot\sqrt{(R+M_f/a_R+u)(R+M_f/a_R+v)}\\
		&\qquad\qquad\qquad\qquad\qquad\qquad\cdot e^{-4(R+M_f/a_R+u)(R+M_f/a_R+v)\sin^2{\frac{t}{2\sqrt{(R+M_f/a_R+u)(R+M_f/a_R+v)}}}}\mathrm{d}t\mathrm{d}u\mathrm{d}v\\
		&\sim\frac{2}{\sqrt{\pi}}|f(M_f^-)|^2\int_{0}^{+\infty}\int_{-\infty}^{0}e^{-(u-v)^2}\mathrm{d}u\mathrm{d}v\cdot(R+M_f/a_R).
	\end{split}
\end{equation*}
	
The above calculations yield that when $R\to+\infty$,
	\begin{equation*}
		\begin{split}
			\mathbb{E}\Big[\sum_{z\in\mathcal{G}}\big|f\big(a_R(|z|-R)\big)\big|\Big]=O\left(\left(\Var\Big(\sum_{z\in\mathcal{G}}f\big(a_R(|z|-R)\big)\Big)\right)^2\right),
		\end{split}
	\end{equation*}
and then we can use Soshnikov's Theorem~\ref{Soshnikov} to get the central limit theorem.

This completes the proof of Theorem~\ref{Small White Noise Ginibre}.
\end{proof}

%
%\section{Appendix: Finite Coulomb gases}
%
%
%Versions of the results above can be also obtained in the context
%of determinantal Coulomb gases with $n$ particles as $n$ goes to infinity.
%In that case, the analogue of 
%Theorem 1.1 and Theorem 1.2 would be close to the following.
%
%Let $V:[0,\infty)\to \mathbb R$ be a continuous function such that
%\begin{itemize}
%\item $V(r) \geq \log r$ for every $r \in [0,1)$,
%\item $V(r)=\log r + (r-1)^2 + o(r-1)^2$ as $r \to 1$,
%\item $V(r) > \log r$ for every $r \in (1,\infty)$ and
%\item $\liminf_{r \to \infty} \frac{V(r)}{\log r}>1 $
%\end{itemize}


\begin{thebibliography}{99}
	\bibitem{Bo} A. Borodin, {\it Determinantal point processes}. Oxford Handbook of Random Matrix Theory, Oxford Univ. Press, Oxford, (2011).
	
	\bibitem{Bu} A. I. Bufetov, {\it The conditional measures for the determinantal point process with the Bergman kernel}. arXiv:2112.15557 (Dec 2021).
	
	\bibitem{BD} A. I. Bufetov, A. V. Dymov, {\it A functional limit theorem for the sine-process}. Int. Math. Res. Not. IMRN 2019, no. 1, 249-319.
	
	\bibitem{BQ} A. I. Bufetov, Y. Qiu, {\it Determinantal point processes associated with Hilbert spaces of holomorphic functions}. Comm. Math. Phys. 351 (2017), 1-44.
	
	\bibitem{DV}
	D. J. Daley, D. Vere-Jones, {\it An introduction to the theory of point processes. Vol. II}. General theory and structure. Second edition. Probability and its Applications (New York). Springer, New York, 2008.
	
	\bibitem{Gh} S. Ghosh, {\it Determinantal processes and completeness of random exponentials: the critical case}. Probab. Theory Related Fields 163 (2015), 643-665.
	
	\bibitem{GP} S. Ghosh, Y. Peres, {\it Rigidity and tolerance in point processes: Gaussian zeros and Ginibre eigenvalues}. Duke Math. J. 166 (2017), 1789-1858.
	
	\bibitem{Gi} J. Ginibre, {\it Statistical ensembles of complex, quaternion, and real matrices}. J. Mathematical Phys. 6 (1965), 440-449.
	
	\bibitem{HKPV} J. B. Hough, M. Krishnapur, Y. Peres, B. Vir\'ag, {\it Zeros of Gaussian analytic functions and determinantal point processes}. University Lecture Series, 51. American Mathematical Society, Providence, RI, 2009.
	
	\bibitem{KK} Y. G. Kondratiev, T. Kuna, {\it Harmonic analysis on configuration space. I. General theory}. Infin. Dimens. Anal. Quantum Probab. Relat. Top. 5 (2002), 201-233.
	
	\bibitem{Ko} E. Kostlan, {\it On the spectra of Gaussian matrices}. Linear Algebra Appl. 162/164 (1992), 385-388. Directions in matrix theory (Auburn, AL, 1990).
	
	\bibitem{Kr} M. Krishnapur, {\it From random matrices to random analytic functions}. Ann. Probab. 37 (2009), 314-346
	
	
	\bibitem{Le} A. Lenard, {\it Correlation functions and the uniqueness of the state in classical statistical mechanics}. Comm. Math. Phys. 30 (1973), 35-44.
	
	\bibitem{Ma} O. Macchi, {\it The coincidence approach to stochastic point processes}. Advances in Appl. Probability. 7 (1975), 83-122.
	
	\bibitem{PV} Y. Peres, B. Vir\'ag, {\it Zeros of the i.i.d. Gaussian power series: A conformally invariant determinantal process}. Acta Math. 194 (2005), 1-35.
	
	\bibitem{ST} T. Shirai, Y. Takahashi, {\it Fermion process and Fredholm determinant}. Proceedings of the Second ISAAC Congress, Vol. 1 (Fukuoka, 1999), 15-23. Int. Soc. Anal. Appl. Comput., 7, Kluwer Acad. Publ., Dordrecht, 2000.
	
	\bibitem{So} A. Soshnikov, {\it Determinantal random point fields}. Russian Math. Surveys. 55 (2000), 923-975.
	
	\bibitem{Sos} A. Soshnikov, {\it Gaussian limit for determinantal random point fields}. Ann. Probab. 30 (2002), 171-187.
\end{thebibliography}
\end{document}